\documentclass{amsart}
\usepackage [latin1]{inputenc}
\usepackage{hyperref}
\def\romega{{r_{\phantom{a}\!\!\!\!_\Omega}}}
\def\rgamma{{r_{\phantom{a}\!\!\!\!_\Gamma}}}
\newcommand\R{\mathbb R}
\newcommand\N{\mathbb N}
\newcommand\Z{\mathbb Z}
\newcommand\loc{{\text{\upshape loc}}}
\newcommand\negquad{\!\!\!\!}

\newcommand{\vertiii}[1]{{\left\vert\kern-0.25ex\left\vert\kern-0.25ex\left\vert #1
    \right\vert\kern-0.25ex\right\vert\kern-0.25ex\right\vert}}
\DeclareMathOperator\Tr{Tr}
\numberwithin{equation}{section}
\newtheorem{thm}{Theorem}
\numberwithin{thm}{section}

\numberwithin{cor}{section}
\newtheorem{lem}{Lemma}
\numberwithin{lem}{section}
\newtheorem{prop}{Proposition}
\numberwithin{prop}{section}
\theoremstyle{definition}
\newtheorem{definition}{Definition}

\numberwithin{definition}{section}
\theoremstyle{remark}
\newtheorem{rem}{Remark}
\numberwithin{rem}{section}

\begin{document}
\title[Nontrivial solutions for a class of semilinear...]
{Nontrivial solutions for a class of semilinear elliptic equations with a nonlinear Goldstein--Wentzell boundary condition}
\author{Enzo Vitillaro}
\address[E.~Vitillaro]
       {Dipartimento di Matematica e Informatica, Universit\`a di Perugia\\
       Via Vanvitelli,1 06123 Perugia ITALY}
\email{enzo.vitillaro@unipg.it}




\thanks{
The research was supported by Universit\`a degli Studi di Perugia''.}

\begin{abstract} The paper deals with the existence and multiplicity of nontrivial solutions for the doubly elliptic problem
$$\begin{cases} -\Delta u=f(u) \qquad &\text{in
$\Omega$,}\\
\phantom{-}u=0 &\text{on $\Gamma_0$,}\\
-\Delta_\Gamma u +\partial_\nu u =g(u)\qquad
&\text{on
$\Gamma_1$,}
\end{cases}
$$
where $\Omega$ is a bounded open domain of $\R^N$ ($N\ge
2$) with $C^1$ boundary  $\Gamma=\partial\Omega$, with $\Gamma=\Gamma_0\cup\Gamma_1$, $\Gamma_0\cap\Gamma_1=\emptyset$,
$\Gamma_1$ being nonempty and relatively open on $\Gamma$,  $\mathcal{H}^{N-1}(\Gamma_0)>0$. The terms $f$ and $g$ are subcritical with respect to Sobolev embeddings, respectively  in $\Omega$ and  on $\partial\Omega$.

We prove that, under suitable assumptions, the problem admits nontrivial solutions at the  depth of the potential well  energy level, which is the minimum energy level for nontrivial solutions. We also  prove that the problem has infinitely many solutions at higher energy levels.
\end{abstract}

\maketitle
\section{Introduction and main results}
\subsection{Presentation of the problem and literature overview}\label{intro}
We deal with the doubly elliptic problem
\begin{equation}\label{1.1}
\begin{cases} -\Delta u=f(u) \qquad &\text{in
$\Omega$,}\\
\phantom{-}u=0 &\text{on $\Gamma_0$,}\\
-\Delta_\Gamma u +\partial_\nu u =g(u)\qquad
&\text{on
$\Gamma_1$,}
\end{cases}
\end{equation}
where $\Omega$ is a bounded open domain of $\R^N$ ($N\ge
2$) with $C^1$ boundary (see \cite{grisvard}). We denote  $\Gamma=\partial\Omega$ and we assume
$\Gamma=\Gamma_0\cup\Gamma_1$, $\Gamma_0\cap\Gamma_1=\emptyset$,
$\Gamma_1$ being nonempty and relatively open on $\Gamma$ (or equivalently $\overline{\Gamma_0}=\Gamma_0$).
Denoting by
$\mathcal{H}^{N-1}$ the  Hausdorff  measure, we assume that   $\mathcal{H}^{N-1}(\overline{\Gamma}_0\cap\overline{\Gamma}_1)=0$ and $\mathcal{H}^{N-1}(\Gamma_0)>0$.
These
properties of $\Omega$, $\Gamma_0$ and $\Gamma_1$ will be assumed,
without further comments, throughout the paper.

Moreover, in \eqref{1.1}, we respectively denote by $\Delta$ and $\Delta_\Gamma$ the Laplace and the Laplace--Beltrami operators, while $\nu$ stands for the outward unit normal to $\Omega$.

The terms $f$ and $g$ respectively represent an internal and a boundary sources. The specific assumptions concerning them will be presented in the sequel.

Semilinear elliptic equations with nonlinear Neumann boundary conditions, such as problem \eqref{1.1}, without the Laplace--Beltrami term, have a wide literature. Without any aim of completeness, here we refer to \cite{adyadava91,atkinson,BenAyed,Bonder,ChipotShafrir,ChlebikFilaReichel,
delPino,inkmann,Pierotti,quitreich,Terracini,tsungfang}.

Boundary conditions like the one in \eqref{1.1}, but without the nonlinear sources $f$ and $g$, are known in the literature as generalized Wentzell (sometimes spelled as Vencel) or Goldstein--Wentzell boundary conditions. They have appeared in several papers dealing with linear evolutions problems. See for example \cite{CFGGGER,grecoviglialoro,quarteroni,nicaisemazzucato,Rom2014DCDS,vazvitHLB} and \cite{lionstata}, to which we refer for  the physical motivations of this kind of problems.

The same boundary condition also appears in the context of coupled bulk--surface elliptic systems. See for example \cite{Dambrine,DuWangXia,ElliottRanner,KnopfLiu,XiaWang}, and the recent papers \cite{Eigenvalues,EigenvaluesCorrigendum} by the author. All the papers above deal with linear eigenvalues problems related to the Wentzell boundary condition. We also  refer to \cite{LiSu} (also giving a physical derivation of the boundary condition) and the references therein.

On the other hand, to the author's knowledge, a Goldstein--Wentzell boundary condition has been studied only in connection with the Laplace equation  and when $g(u)=|u|^{q-2}u$, $q>2$, in the recent papers \cite{MP,MPCorrigendum} by the author.

 The motivation for studying the problem \eqref{1.1} comes from a series of papers by the author. They deal with the wave equation with hyperbolic dynamical boundary conditions, internal and  boundary damping and source terms. Assuming in the present paper for simplicity  that the terms $f$ and $g$ are independent on $x$, and taking the damping terms in their simplest  form, these papers concern the evolutionary boundary value problem
\begin{equation}\label{1.2}
\begin{cases} u_{tt}-\Delta u+\alpha |u_t|^{m-2}u_t=f(u) \qquad &\text{in
$(0,\infty)\times\Omega$,}\\
u=0 &\text{on $(0,\infty)\times \Gamma_0$,}\\
u_{tt}+\partial_\nu u-\Delta_\Gamma
u+\beta |u_t|^{\mu-2}u_t=g(u)\qquad &\text{on
$(0,\infty)\times \Gamma_1$,}
\end{cases}
\end{equation}
where $u=u(t,x)$, $t>0$, $x\in\Omega$, $\Delta=\Delta_x$ and $\Delta_\Gamma=(\Delta_\Gamma)_x$ denote the operators in \eqref{1.1} \emph{with respect to the space variable}, and where we take
\begin{equation}\label{1.3}
\alpha,\beta\ge 0, \quad\text{and}\quad m,\mu>1.
\end{equation}
The  initial value problem associated with \eqref{1.2} was introduced in \cite{AMS} when $\alpha=0$, $f\equiv 0$, $\beta=1$ and $g(u)=|u|^{q-2}u$, $q\ge 2$. Subsequently, it was studied, as a particular case, in \cite{Dresda1}--\cite{Dresda3}. We refer to \cite{goldsteingisele,Dresda1} for the physical derivation of \eqref{1.2}. When $N=2$ it describes the vibrations of a membrane with a part of the boundary carrying a linear density of  kinetic energy.

Several papers in the literature deal with the wave equation with hyperbolic boundary conditions like \eqref{1.2} or, more in general, with kinetic boundary conditions, which are like \eqref{1.2} without the Laplace--Beltrami term. A complete list of references for this kind of problems would exceed our aim, here we shall refer to \cite{bociuNAMA,bociulasiecka1,CDCL,CDCM,chueshovellerlasiecka,FGGGR,fisvit2,fisvit2corrigendum,lastat,vazvitM3AS,global,
xiaoliang2,xiaoliang1, zhang,zuazua}.

In this paper we shall deal with the application of the well--known potential--well theory, originally introduced in \cite{paynesattinger,sattinger1}, to the case of two independent source terms, one in the interior and the other one at the boundary. This type of arguments has been used, in absence of the Laplace--Beltrami term, in the already quoted paper \cite{CDCL}, when $f$ and $g$ are of algebraic type, without using an optimal  potential well.

The aim of this paper is to introduce an optimal depth for it, considering terms $f$ and $g$ as general as possible, and to characterize it in several ways. In this regard, it is useful to prove that \eqref{1.2} possesses nontrivial stationary solutions, which turn out to be solutions of \eqref{1.1}, exactly at this energy level. Moreover, planning to get in a forthcoming paper clear--cut criteria on the initial data to discriminate between global existence and blow--up for solutions of \eqref{1.2}, it is also useful to characterize sets of initial data which automatically solve \eqref{1.1}.

This goal was achieved in \cite{MP,MPCorrigendum} in the simplest case $f\equiv 0$, $g(u)=|u|^{q-2}u$, $q>2$, while in this paper we shall deal with a more involved case.

\subsection{Main assumptions and functional setting}\label{section 1.2}
 We  assume $f$ to be subcritical with respect to the  Sobolev Embedding $H^1(\Omega)\hookrightarrow L^\tau(\Omega)$, while
$g$ enjoys the same property with respect to $H^1(\Gamma)\hookrightarrow L^\tau(\Gamma)$. The critical exponents of these embeddings respectively are
 \begin{equation}\label{1.4}
 \romega=
 \begin{cases}
\tfrac {2N}{N-2} &\text{if $N \ge 3$,}\\ \phantom{aa}\infty &\text{if $N=2$},
\end{cases}
\qquad \rgamma=
\begin{cases}
\tfrac {2(N-1)}{N-3} &\text{if $N \ge 4$,}\\ \phantom{aaa}\infty &\text{if $N=2,
3$}.
\end{cases}
 \end{equation}
Our assumptions of $f$ and $g$ will allow them to have a linear behavior close to the origin, this feature being a novelty for this kind of problems. This behavior will be related to the first eigenvalue $\lambda_1$ of the doubly elliptic eigenvalue problem
\begin{equation}\label{1.5}
\begin{cases} -\Delta u=\lambda u \qquad &\text{in
$\Omega$,}\\
\phantom{-}u=0 &\text{on $\Gamma_0$,}\\
-\Delta_\Gamma u +\partial_\nu u =\lambda u\qquad
&\text{on
$\Gamma_1$,}
\end{cases}
\end{equation}
which has been studied by the author in \cite{Eigenvalues,EigenvaluesCorrigendum}. In particular, since $\Omega$ is connected, \cite[Theorem~1.1]{Eigenvalues} applies and $\lambda_1$ is positive.

The main assumptions we made on the couple $(f,g)$ are the following ones. In them $u^+$ and $u^-$ will respectively denote the positive and negative part of $u\in\R$:
\renewcommand{\labelenumi}{{(A\arabic{enumi})}}
\begin{enumerate}
\item $f,g\in C(\R)\cap C^1([0,\infty))\cap C^1((-\infty,0])$, and there are $2<p<\romega$,  $2<q<\rgamma$ such that
$|f'(u)|=O(|u|^{p-2})$,  $|g'(u)|=O(|u|^{q-2})$ as $|u|\to\infty$;
\item the functions $f(u)/u$ and $g(u)/u$ are both decreasing in $(-\infty,0)$ and increasing in $(0,\infty)$.
Moreover, denoting $\ell^\pm=\lim\limits_{u\to 0^\pm} f(u)/u$ and $\mathfrak{m}^\pm=\lim\limits_{u\to 0^\pm} g(u)/u$, we have
\begin{equation}\label{1.6}
  -\infty<\ell^\pm,\,\,\mathfrak{m}^\pm<\lambda_1;
\end{equation}
\item  setting
\begin{equation}\label{1.7}
\left\{
\begin{alignedat}{2}
&h(u):=f(u)-(\ell^+u^+-\ell^-u^-),\qquad && H(u)=\int_0^u h(s)ds,  \\
&k(u):=g(u)-(\mathfrak{m}^+u^+-\mathfrak{m}^-u^-), \quad && K(u)=\int_0^u k(s)ds,
\end{alignedat}\right.\quad\text{for all $u\in\R$,}
\end{equation}
 there are constants $p_0,q_0>2$ and $M\ge 0$ such that
\begin{equation}\label{1.8}
h(u)u\ge p_0 H(u),\quad\text{and}\quad k(u)u\ge q_0 K(u)\quad\text{for $|u|\ge M$;}
\end{equation}
\item the functions $\sigma(u):=h(u)/u$ and $\xi(u):=k(u)/u$ have no critical points at positive level, and
\begin{equation}\label{1.9}
\lim_{u\to\pm\infty}\sigma(u)>0\qquad\text{or}\quad \lim_{u\to\pm\infty}\xi(u)>0.
\end{equation}
\end{enumerate}

The simplest example of a couple $(f,g)$ satisfying assumptions (A1--4) is given by
\begin{equation}\label{1.10}
\begin{gathered}
f(u)=\gamma |u|^{p-2}u, \qquad \text{and}\quad  g(u)=\delta |u|^{q-2}u,\quad\text{when}  \\
2<p<\romega, \quad 2<q<\rgamma, \quad \gamma,\delta\ge 0, \quad \gamma+\delta>0.
\end{gathered}
\end{equation}
Indeed, assumptions (A1--2) trivially hold with $\ell^\pm=\mathfrak{m}^\pm=0$, so $h\equiv f$ and $k\equiv g$ in \eqref{1.7}. Hence
also \eqref{1.8} holds, with assumption (A3), by taking $p_0=p$ and $q_0=q$. Finally, $\sigma(u)=\gamma |u|^{p-2}$ and
$\xi(u)=\delta |u|^{q-2}$ trivially satisfy assumption (A4), since we have $\gamma>0$ or $\delta>0$.

A more general example is the following couple $(f,g)$, which essentially generalizes \eqref{1.10} and constitutes the model example for this paper:
\begin{equation}\label{1.11}
\begin{gathered}
f(u)=\ell^+u^+-\ell^-u^-+\gamma_0 |u|^{p_0-2}u+\gamma |u|^{p-2}u, \qquad 2<p_0<p<\romega,  \\
g(u)=\mathfrak{m}^+u^+-\mathfrak{m}^-u^-+\delta_0 |u|^{q_0-2}u+\delta |u|^{q-2}u, \qquad 2<q_0<q<\rgamma,  \\
\gamma_0,\gamma,\delta_0,\delta\ge 0, \qquad \gamma_0+\gamma+\delta_0+\delta>0,\qquad \ell^\pm,\mathfrak{m}^\pm <\lambda_1.
\end{gathered}
\end{equation}
Indeed, also in this case assumptions (A1--2) trivially hold, and in \eqref{1.7} we have $h(u)=\gamma_0|u|^{p_0-2}u+\gamma |u|^{p-2}u$ and $k(u)=\delta_0|u|^{q_0-2}u+\delta |u|^{q-2}u$. Since, for example considering $h$, we have $h(u)u-p_0H(u)=\gamma(p-p_0)|u|^p/p\ge 0$, also \eqref{1.8} holds true, together with assumption (A3). Also in this case $h$ and $k$ have no critical points (unless they identically vanish), and \eqref{1.9} holds, with (A4), since $(\gamma_0,\gamma,\delta_0,\delta)\not=0$.

The model couple of nonlinearities $(f,g)$ given in the example \eqref{1.11}, when $\ell^+=\ell^-$ and $\mathfrak{m}^+=\mathfrak{m}^-$, covers as possible the model couple $(f_1,g_1)$ given in example (1.3) of \cite{Dresda1,Dresda2}. The case $\ell^+\not=\ell^-$ and/or $\mathfrak{m}^+\not=\mathfrak{m}^-$ is studied here for the sake of completeness. Moreover, other different examples of couples satisfying assumptions (A1--4)  may be given. One of them is made explicit in the sequel, see \S~\ref{section 2.5}.

It is worth to make some comments on assumptions (A1--4). Although the main novelty of problem \eqref{1.1} is constituted by the presence of two completely independent source terms, a short comparison of assumptions (A1--4) with similar sets of assumptions in the literature, in the case  $g\equiv 0$, can be of some interest.

Being potential--well arguments of wide use, any comparison of this type would be largely incomplete. Here we shall refer to the classical papers \cite{AmbrosettiRabinowitz} (as generalized in \cite[Chapter~8, \S 8.2, p. 123]{ambrosettimalchiodi}), \cite{AmbrosettiLincei} (as generalized in \cite[Chapter~7, \S 7.62, p. 110]{ambrosettimalchiodi}), and \cite{paynesattinger}.

Assumption (A1) is a classical regularity and growth assumption, which is present in a weaker form in the first paper quoted above, and in a stronger one in the others two papers. The monotonicity of the function $f(u)/u$, assumed in (A2), is the key ingredient in the last two papers quoted above, while it is absent in the first one. Moreover, assumption (A3) allows $f$ to have a linear behavior close to the origin, as in \cite[Chapter~8, \S 8.2, p. 123]{ambrosettimalchiodi}, while in the other two papers above one assumes that $f(u)=o(|u|)$ as $u\to 0$. Assumption (A3) is the classical Ambrosetti--Rabinowitz condition for $h$ and $k$. This condition  is explicitly assumed in the first and the third papers above and it is implicitly assumed in the second one. Indeed, by the convexity of $\sigma$ assumed in \cite[Chapter~7, \S 7.62, p. 110]{ambrosettimalchiodi}, the authors get (see \cite[Chapter~7, proof of Lemma~7.17, p.112]{ambrosettimalchiodi}) that $f(u)u\ge 3F(u)$ for all $u\in\R$.
The lack of critical points of $\sigma$, which is  prescribed by assumption (A4), is explicitly assumed in the last two papers above, while it is not present in the first one. Finally, the condition \eqref{1.9} allows to exclude the case $f\equiv0$, equally excluded in all papers above.

Summarizing, the set (A1--4) is more restrictive than the one in \cite[Chapter~8, \S 8.2, p. 123]{ambrosettimalchiodi}, in which merely the existence of critical points at the Mountain Pass level is obtained. This fact seems quite natural, since we are going to give more precise results.
On the other hand, the set (A1--4) is less restrictive than those in the last two papers.

\begin{rem}\label{Remark0} Assumption (A4) needs a further explanation. Let us  make some remarks on the functions $\sigma$ and $\xi$. By \eqref{1.7} we have
\begin{equation}\label{1.12}
\sigma(u)=\begin{cases}
f(u)/u-\ell^+,\,\,&\text{if $u>0$,}\\
f(u)/u-\ell^-,\,\,&\text{if $u<0$,}
\end{cases}\quad
\xi(u)=\begin{cases}
g(u)/u-\mathfrak{m}^+,\,\,&\text{if $u>0$,}\\
g(u)/u-\mathfrak{m}^-,\,\,&\text{if $u<0$.}
\end{cases}
\end{equation}
Hence, by (A2), the functions $\sigma$ and $\xi$ extend (keeping the same notation) to $\sigma,\xi\in C(\R)$, provided we set $\sigma(0)=\xi(0)=0$. Using assumption (A2) again, we get that these functions are decreasing in $(-\infty,0]$ and increasing in $[0,\infty)$, so $\sigma,\xi\ge 0$ in $\R$.

Moreover, when considering the possible behaviors of $\sigma$ in $[0,\infty)$, taking into account the monotonicity of $\sigma$, the following alternative holds:
\renewcommand{\labelenumi}{{\roman{enumi})}}
\begin{enumerate}
\item either $\sigma\equiv 0$ in $[0,\infty)$, this fact being equivalent to $\lim\limits_{u\to\infty} \sigma(u)=0$,
\item or $\sigma\not\equiv 0$ in $[0,\infty)$, this fact being equivalent to $\lim\limits_{u\to\infty} \sigma(u)>0$. In this case
there is $\overline{u}\ge 0$ such that $\sigma=0$ in $[0,\overline{u}]$ and $\sigma>0$ in $(\overline{u},\infty)$.
\end{enumerate}
The lack of critical points asserted in assumption (A4) then reduces to assert that, when the alternative ii) holds, one has $\sigma'>0$ in $(\overline{u},\infty)$. This strict  monotonicity of $\sigma$ is explicitly assumed in the last two papers quoted above
in the whole of $(0,\infty)$.

When considering the possible behaviors of $\sigma$ in $(-\infty,0]$ and those of $\xi$ in $[0,\infty)$ and in $(-\infty,0]$, taking into account the monotonicity properties above, alternatives analogous to i)--ii) above occur, and the lack of critical points asserted in assumption (A4) has the same meaning.
\end{rem}
When dealing with problem \eqref{1.2}, we shall also assume that
\begin{equation}\label{1.13}
\begin{alignedat}2
 &p\le
\begin{cases}
1+\romega/2&\text{if  $\alpha=0$,}\\
1+\romega/ \overline{m}'&\text{if  $\alpha>0$,}
\end{cases}\quad
&&\text{where}\quad \overline{m}:=\max\{2,m\},\\
&q\le \begin{cases}
1+\rgamma/2&\text{if  $\beta=0$,}\\
1+\rgamma/ \overline{\mu}'&\text{if  $\beta>0$,}
\end{cases}\quad
&&\text{where}\quad \overline{\mu}:=\max\{2,\mu\}.
\end{alignedat}
\end{equation}
This assumption allows to properly define weak solutions of \eqref{1.2} and it is related with well--posedness issues, see \cite{Dresda1,Dresda2}.
 \begin{footnote}{Clearly, assumption \eqref{1.13} has to be skipped when dealing with stationary solutions of \eqref{1.2}.}\end{footnote}

  We point out  that, although also the cases $p\ge \romega$ (when $N\ge 3$) and $q\ge \rgamma$ (when $N\ge 3$) were considered in \cite{Dresda2}, only the cases $p<\romega$ and $q<\rgamma$ (considered in this paper) are of interest when dealing with the dichotomy between global existence and blow--up, see \cite[Remark~1, p.4580]{Dresda3}.

In the sequel we shall use the primitives of $f$ and $g$, respectively defined by
\begin{equation}\label{1.14}
F(u)=\int_0^u f(s)\,ds,\qquad\text{and}\quad G(u)=\int_0^u g(s)\,ds\qquad\text{for $u\in\R$.}
\end{equation}
 To state our main results we first introduce some basic notation. Subsequently, we shall
identify $L^\vartheta(\Gamma_1)$, for $1\le \vartheta\le \infty$, with its isometric image in $L^\vartheta(\Gamma)$,
that is
\begin{equation}\label{1.15}
L^\vartheta(\Gamma_1)=\{u\in L^\vartheta(\Gamma): u=0\quad\text{a.e. on $\Gamma_0$}\},
\end{equation}
where $L^\vartheta(\Gamma)$ and a.e. equivalence are meant in the sense of the Hausdorff measure $\mathcal{H}^{N-1}$ (restricted to measurable subsets of $\Gamma$).

Moreover we shall denote by $\Tr$ the trace operator from $H^1(\Omega)$ onto $H^{1/2}(\Gamma)$ and, for simplicity of notation,
$\Tr u=u_{|\Gamma}$ for all $u\in H^1(\Omega)$.

We introduce the Hilbert spaces $H^0 = L^2(\Omega)\times L^2(\Gamma_1)$ and
\begin{equation}\label{1.16}
H^1 = \left\{(u,v)\in H^1(\Omega)\times H^1(\Gamma): v=u_{|\Gamma}, v=0
\quad \text{a.e. on $\Gamma_0$}\right\},
\end{equation}
with the topologies inherited from the product spaces. For the sake of
simplicity we shall identify, when useful, $H^1$ with its isomorphic
counterpart
\begin{equation}\label{1.17}
H^1_{\Gamma_0}(\Omega,\Gamma)=\{u\in H^1(\Omega): u_{|\Gamma}\in H^1(\Gamma)\cap
L^2(\Gamma_1)\},
\end{equation}
which has been studied, for example, in \cite{pucvit}, through the identification $(u,u_{|\Gamma})\mapsto
u$. We shall write, without further comments, $u\in H^1$ for
functions defined on $\Omega$. Moreover, we shall drop the notation
$u_{|\Gamma}$, when useful, so we shall write $\|u\|_{L^2(\Gamma)}$
and so on, for $u\in H^1$.
We shall also drop the notation  $d\mathcal{H}^{N-1}$ in boundary integrals, so writing $\int_\Gamma u=\int_\Gamma u\,d\mathcal{H}^{N-1}$.
\subsection{Main results}\label{Section 1.3}
By assumption (A1) (see Lemma~\ref{Lemma 4} below)  we can introduce in $H^1$ the nonlinear functional $I\in C^1(H^1)=C^1(H^1;\R)$ defined by
\begin{footnote}{here $\nabla_\Gamma$
denotes the Riemannian gradient on $\Gamma$ and $|\cdot|_\Gamma$,
the norm associated to the Riemannian scalar product on the tangent
bundle of $\Gamma$. See \S~\ref{section 2.2}.}\end{footnote}
\begin{equation}\label{1.18}
 I(u)=\tfrac 12 \int_\Omega |\nabla u|^2 +\tfrac 12
\int_{\Gamma_1} |\nabla_\Gamma  u|_\Gamma^2-\int_\Omega F(u)-\int_{\Gamma_1}G(u),
\end{equation}
which represents the potential energy associated to problem \eqref{1.2}. For this reason we shall call it the \emph{energy functional} when dealing with  problem \eqref{1.1}.

We also introduce the potential--well depth $d$ given by
\begin{equation}\label{1.19}
 d=\inf_{u\in H^1\setminus\{0\}}\sup_{\lambda>0} I(\lambda u).
\end{equation}

Our first main result, which is a main tool in the subsequent discussion, asserts that problem \eqref{1.1} admits  at least a nontrivial weak solution, see Definition~\ref{Definition 2} below, coinciding with a critical point of the functional $I$, at the \emph{positive} energy level $d$.  We also recognize that, when \eqref{1.13} holds, such a solution is also a stationary weak solution of \eqref{1.2}, see Definition~\ref{Definition 1} below.

\begin{thm}[\bf Existence of solutions]\label{Theorem 1} Let assumptions (A1--4) hold. Then problem \eqref{1.1} has at least a weak solution $u\in H^1$ such that $I(u)=d>0$. When also \eqref{1.13} holds, $u$ is  also a stationary weak solution of problem \eqref{1.2}.
Moreover $d$ coincides with the Mountain Pass level of the functional $I$, that is $d=c$, where
\begin{equation}\label{1.19bis}
c:=\inf_{\sigma\in\Sigma} \max_{t\in [0,1]}I(\sigma(t)), \Sigma=\{\sigma\in C([0,1];H^1):\,\sigma(0)=0, I(\sigma(1))<0\}.
\end{equation}
\end{thm}
The proof of Theorem~\ref{Theorem 1} relies on applying  a  variant of the Mountain Pass Theorem, recalled in \S~\ref{section 2}, and exploiting the consequences of assumption (A2).

Weak solutions of \eqref{1.1} at level $d$ can be further characterized by introducing the Nehari functional $\mathcal{K}\in C^1(H^1)$ defined by
\begin{equation}\label{1.20}
 \mathcal{K}(u)=\int_\Omega |\nabla u|^2 + \int_{\Gamma_1} |\nabla_\Gamma  u|_\Gamma^2-\int_\Omega f(u)u-\int_{\Gamma_1}g(u)u,
\end{equation}
and the Nehari manifold of the functional $I$, that is
\begin{equation}\label{1.21}
 \mathcal{N}:=\{u\in H^1\setminus\{0\}: \mathcal{K}(u)=0\}.
 \end{equation}
\begin{thm}[\bf Characterizations of $d$ and of solutions at level $d$]\label{Theorem 2}
 Let assumptions (A1--4)  hold. Then we have
\begin{equation}\label{1.22}
d=\inf_{u\in\mathcal{N}}I(u),
\end{equation}
and, consequently, weak solutions $u$ of \eqref{1.1} such that $I(u)=d$ are lowest energy nontrivial weak solutions of \eqref{1.1}, hence they satisfy the condition
\begin{equation}\label{1.23}
u\in\mathcal{N},\quad I(u)\le d.
\end{equation}
Conversely, any $u\in H^1$ satisfying \eqref{1.23} is a lowest energy nontrivial weak solution of \eqref{1.1} and, when also \eqref{1.13} holds, it is a weak stationary solution of \eqref{1.2}.
\end{thm}
The proof of  the minimality of the energy of solutions  at level $d$, stated in  Theorem~\ref{Theorem 2}, is of  elementary nature.
On the other hand, the last part of Theorem~\ref{Theorem 2} is of particular interest when discussing the long--time behavior of weak solutions of \eqref{1.2} when the initial datum $u_0:= u(0,\cdot)$ satisfies \eqref{1.23}

To illustrate the interest of  the minimality asserted in Theorem~\ref{Theorem 2}, since there are solutions at higher energy levels, we are now going to present our third main result, which can be also of independent interest.

\begin{thm}[\bf Multiplicity]\label{Theorem 3} Let assumption (A1--4) hold, and also suppose that $f$ and $g$ are odd. Then there is a sequence $(u_n)_n$ in $H^1$ such that $u_n$ and $-u_n$ are nontrivial weak solutions of \eqref{1.1} with  $I(u_n)=I(-u_n)\to\infty$ as $n\to\infty$.
\end{thm}
Theorem~\ref{Theorem 3} will be proved by carefully distinguishing between two different cases, in which different arguments have to be used. To understand them, we point out that, since $f$ and $g$ are odd, the functions $f(u)/$ and $g(u)/u$ in assumption (A2) are even. Consequently, in this case we have
\begin{equation}\label{1.24}
\ell:=\ell^+=\ell^-<\lambda_1,\qquad\text{and}\quad \mathfrak{m}:=\mathfrak{m}^+=\mathfrak{m}^-<\lambda_1.
\end{equation}
Hence the functions $h$ and $k$ in \eqref{1.7} are odd, while $\sigma$ and  $\xi$ in assumption (A4) are even. Consequently we have
$\lim\limits_{u\to\infty} \sigma(u)=\lim\limits_{u\to -\infty} \sigma(u)$ and $\lim\limits_{u\to\infty} \xi(u)=\lim\limits_{u\to -\infty} \xi(u)$. Hence assumption \eqref{1.9} simplifies to asking that $\lim\limits_{|u|\to\infty} \sigma(u)>0$
or $\lim\limits_{|u|\to\infty} \xi(u)>0$.

We can then consider the following alternative: either $\lim\limits_{|u|\to\infty} \sigma(u)>0$,
  or $\sigma\equiv h\equiv H\equiv 0$ in $\R$ and $\lim\limits_{|u|\to\infty} \xi(u)>0$.
In the first case we are going to apply  the $\Z_2$--version of the Mountain Pass Theorem to the functional $I$, while in the second one we are going to apply it in a different variational setting, as done in \cite{MP} in a simpler case. See \S~\ref{section 4.2} below.

Finally, we are going to further characterize the potential--well depth $d$, and the weak solutions at this energetic level, when the couple $(f,g)\in C(\R^2;\R^2)$ is odd and positively homogeneous. It is straightforward to recognize that these two further assumptions on $(f,g)$ reduces to assuming that $f$ and $g$ are given by \eqref{1.10}, with $\delta=0$, or $\gamma=0$, or $\gamma,\delta>0$ and $p=q$. Since the statement we are going to presents depends on the specific case occurring, we formalize the combination between
(A1--4) and these two further assumptions  as follows.
\renewcommand{\labelenumi}{{(A5)}}
\begin{enumerate}
\item  We suppose that one of the following alternative assumptions holds:
\renewcommand{\labelenumii}{{(A5.\arabic{enumii})}}
\begin{enumerate}
\item $f(u)=\gamma |u|^{p-2}u$, $\gamma>0$, $2<p<\romega$, and $g\equiv 0$;
\item $f\equiv 0$, $g(u)=\delta |u|^{q-2}u$, $\delta>0$, and $2<q<\rgamma$;
\item $f(u)=\gamma |u|^{p-2}u$, \quad $g(u)=\delta |u|^{p-2}u$, $\gamma, \delta>0$, and $2<p<\romega$.
\end{enumerate}
\end{enumerate}
As we already noticed, assumption (A5) yields (A1--4).  We also notice that, since $\romega<\rgamma$, assumption (A5.3) is in agreement with assumption (A1), since  $p=q<\rgamma$.

Before distinguishing the three cases above, we notice that
\begin{equation}\label{1.25}
\|u\|_{H^1}^2:=\int_\Omega |\nabla u|^2+\int_{\Gamma_1} |\nabla_\Gamma u|_\Gamma^2,\qquad u\in H^1,
\end{equation}
defines on $H^1$ a norm $\|\cdot\|_{H^1}$ equivalent to the norm inherited by the space $H^1$ from the product $H^1(\Omega)\times H^1(\Gamma_1)$, see Lemma~\ref{Lemma 1} below.

When assumption (A5.1) holds, we introduce the norm of the Sobolev Embedding operator $H^1\to L^p(\Omega)$, that is
\begin{equation}\label{1.26}
  B_\Omega:=\sup_{u\in H^1\setminus\{0\}} \dfrac{\|u\|_{L^p(\Omega)}}{\|u\|_{H^1}},
\end{equation}
so that one has
\begin{equation}\label{1.27}
\|u\|_{L^p(\Omega)}\le B_\Omega \|u\|_{H^1}\quad\text{for all $u\in H^1$.}
\end{equation}
\begin{thm}\label{Theorem 4} Let assumption (A5.1) hold and set the positive constants
\begin{equation}\label{1.28}
 \omega_1:=\gamma^{- 1/(p-2)}B_\Omega ^{-p/(p-2)},\qquad\text{and}\quad  \omega_2=B_\Omega\omega_1= \gamma^{- 1/(p-2)}B_\Omega ^{-2/(p-2)}.
\end{equation}
Then we have
\begin{equation}\label{1.29}
 d=\left(\tfrac 12 -\tfrac 1p\right)\omega_1^2=\left(\tfrac 12 -\tfrac 1p\right)\gamma\,\omega_2^p.
\end{equation}
Moreover, if $u$ is a lowest energy nontrivial weak solution of \eqref{1.1}, we  have
\begin{equation}\label{1.30}
\|u\|_{H^1}=\omega_1,\qquad\text{and}\quad \|u\|_{L^p(\Omega)}=\omega_2,
\end{equation}
so that $u$ solves the maximization problem
\begin{equation}\label{1.31}
\max_{u\in H^1\setminus\{0\}}\dfrac{\|u\|_{L^p(\Omega)}}{\|u\|_{H^1}}.
\end{equation}
Conversely,  if $u$ is a solution of \eqref{1.31}, then there is a positive constant $\tau$ such that $v=\tau u$ is a lowest energy nontrivial weak solution of \eqref{1.1}.
\end{thm}

When assumption (A5.2) holds, we introduce the norm of the Trace -- Sobolev  operator $u\mapsto u_{|\Gamma}$ from $H^1$ to $L^q(\Gamma_1)$,  that is
\begin{equation}\label{1.32}
  B_\Gamma:=\sup_{u\in H^1\setminus\{0\}}\dfrac{\|u\|_{L^q(\Gamma_1)}}{\|u\|_{H^1}},
\end{equation}
so that one has
\begin{equation}\label{1.33}
\|u\|_{L^q(\Gamma_1)}\le B_\Gamma \|u\|_{H^1}\quad\text{for all $u\in H^1$.}
\end{equation}
\begin{thm}\label{Theorem 5} Let assumption (A5.2) hold and set the positive constants
\begin{equation}\label{1.34}
 \kappa_1:=\delta^{- 1/(q-2)}B_\Gamma ^{-q/(q-2)},\qquad\text{and}\quad  \kappa_2=B_\Gamma\kappa_1= \delta^{- 1/(q-2)}B_\Gamma ^{-2/(q-2)}.
\end{equation}
Then we have
\begin{equation}\label{1.35}
 d=\left(\tfrac 12 -\tfrac 1q\right)\kappa_1^2=\left(\tfrac 12 -\tfrac 1q\right)\delta \,\kappa_2^q.
\end{equation}
Moreover, if $u$ is a lowest energy nontrivial weak solution of \eqref{1.1}, we  have
\begin{equation}\label{1.36}
\|u\|_{H^1}=\kappa_1,\qquad\text{and}\quad \|u\|_{L^q(\Omega)}=\kappa_2,
\end{equation}
so that $u$ solves the maximization problem
\begin{equation}\label{1.37}
\max_{u\in H^1\setminus\{0\}}\dfrac{\|u\|_{L^q(\Gamma_1)}}{\|u\|_{H^1}}.
\end{equation}
Conversely,  if $u$ is a solution of \eqref{1.37}, then there is a positive constant $\tau$ such that $v=\tau u$ is a lowest energy nontrivial weak solution of \eqref{1.1}.
\end{thm}

When assumption (A5.3) holds, we introduce the product space $X_p:=L^p(\Omega)\times L^p(\Gamma_1)$, endowed with
the norm $\|\cdot\|_{X_p}$  defined by
\begin{equation}\label{1.38}
\|(u,v)\|_{X_p}^p:=\gamma \int_\Omega |u|^p+\delta \int_{\Gamma_1} |v|^p,
\end{equation}
which is trivially equivalent to the standard product  one since $\gamma,\delta>0$. We also introduce the norm of
the Sobolev Embedding operator $H^1\to X_p$, that is
\begin{equation}\label{1.39}
  B_p:=\sup_{u\in H^1\setminus\{0\}} \dfrac{\|(u,u_{|\Gamma})\|_{X_p}}{\|u\|_{H^1}},
\end{equation}
so that one has
\begin{equation}\label{1.40}
\|(u,u_{|\Gamma})\|_{X_p}\le B_p \|u\|_{H^1}\quad\text{for all $u\in H^1$.}
\end{equation}
\begin{thm}\label{Theorem 6} Let assumption (A5.3) hold and set the positive constants
\begin{equation}\label{1.41}
 \zeta_1:=B_p^{-p/(p-2)},\qquad\text{and}\quad  \zeta_2=B_p\zeta_1= B_p ^{-2/(p-2)}.
\end{equation}
Then we have
\begin{equation}\label{1.42}
 d=\left(\tfrac 12 -\tfrac 1p\right)\zeta_1^2=\left(\tfrac 12 -\tfrac 1p\right)\zeta_2^p.
\end{equation}
Moreover, if $u$ is a lowest energy nontrivial weak solution of \eqref{1.1}, we  have
\begin{equation}\label{1.43}
\|u\|_{H^1}=\zeta_1\qquad\text{and}\quad \|(u,u_{|\Gamma})\|_{X_p}=\zeta_2,
\end{equation}
so that $u$ solves the maximization problem
\begin{equation}\label{1.44}
\max_{u\in H^1\setminus\{0\}}\dfrac{\|(u,u_{|\Gamma})\|_{X_p}}{\|u\|_{H^1}}.
\end{equation}
Conversely,  if $u$ is a solution of \eqref{1.44}, then there is a positive constant $\tau$ such that $v=\tau u$ is a lowest energy nontrivial weak solution of \eqref{1.1}.
\end{thm}
\begin{rem}\label{Remark 1}
Clearly, Theorems~\ref{Theorem 4}, \ref{Theorem 5} and \ref{Theorem 6} look as variants of the same result. We present them in this separate form for the sake of clearness. Moreover, the first part of Theorem~\ref{Theorem 5}, when $\delta=1$, clearly reduces to \cite[Theorem~2]{MP}.
\end{rem}

The paper is organized as follows: in Section~\ref{section 2} we shall give all preliminaries needed in the paper.  Section~\ref{section 3} will be devoted to prove Theorems~\ref{Theorem 1} and \ref{Theorem 2}.  In Section~\ref{section 4} we shall prove Theorem~\ref{Theorem 3}, while Theorems~\ref{Theorem 4}, \ref{Theorem 5} and \ref{Theorem 6} will be proved in Section~ \ref{section 5}.

\section{Preliminaries} \label{section 2}
\subsection{Notation.}\label{section 2.1}
We shall adopt the standard notation for
(real) Lebesgue and Sobolev spaces in $\Omega$, referring to
\cite{adams}. For simplicity we shall denote by $\|\cdot\|_{\vartheta}$, for $1\le \vartheta\le \infty$, the norms in $L^\vartheta(\Omega)$ and in $L^\vartheta(\Omega;\R^N)$.

Given a Banach space
$X$ we shall denote by  $X'$  its dual and by $\langle
\cdot,\cdot\rangle_X$ the duality product between them. Moreover, we shall use the standard notation for $X$--valued Lebesgue and Sobolev spaces in a real interval. When another Banach space $Y$ is given we shall denote by
$\mathcal{L}(X,Y)$  the space of bounded linear operators between $X$ and $Y$, and by $\|\cdot\|_{\mathcal{L}(X,Y)}$ the standard norm on it.

\subsection{Function spaces and Riemannian operators on $\Gamma$.} \label{section 2.2}
As we already noticed, Lebesgue spaces on $\Gamma$ and $\Gamma_1$ will be intended with respect to  $\mathcal{H}^{N-1}$, and for simplicity we shall denote, for $1\le \vartheta\le \infty$,
$\|\cdot\|_{\vartheta,\Gamma}=\|\cdot\|_{L^\vartheta(\Gamma)}$ and $\|\cdot\|_{\vartheta,\Gamma_1}=\|\cdot\|_{L^\vartheta(\Gamma_1)}$.

 Sobolev spaces on $\Gamma$ and on its relatively open subsets are classical objects, and  we shall use the standard notation for them. We refer to \cite{grisvard} for their definition in the present case in which $\Gamma$ is merely $C^1$.

  Since $\Gamma$ is $C^1$,  it inherits from $\R^N$ the structure of a Riemannian $C^1$ manifold, see \cite{sternberg}, so in the sequel we shall use some notation of geometric nature, which is quite common when $\Gamma$ is smooth,  see \cite{Boothby, hebey, jost, taylor}, and which can be easily extended to the $C^1$ case, see for example \cite{mugnvit}.  Moreover, since $\Gamma_1$ is relatively open on $\Gamma$, this notation will apply (by restriction) to $\Gamma_1$, without further mention.

We shall denote by $T(\Gamma)$ and $T^*(\Gamma)$ the tangent and cotangent bundles, and  by $(\cdot,\cdot)_\Gamma$ the Riemannian metric inherited from $\R^N$, given in local coordinates by $(u,v)_\Gamma=g_{ij}u^i v^j$ for all $u,v\in T(\Gamma)$
(here and in the sequel the summation convention being in use). The metric induces the fiber--wise defined musical isomorphisms $\flat:T(\Gamma)\to T^*(\Gamma)$ and $\sharp=\flat^{-1}:T^*(\Gamma)\to T(\Gamma)$ defined by
$\langle \flat u,v\rangle_{T(\Gamma)}=(v,u)_\Gamma$ for $u,v\in T(\Gamma)$,
where $\langle \cdot,\cdot\rangle_{T(\Gamma)}$ denotes the fiber-wise defined duality pairing.
The induced bundle metric on $T^*(\Gamma)$, still denoted by $(\cdot,\cdot)_\Gamma$, is then defined  by
the formula $(\omega,\varphi)_\Gamma=\langle \omega,\sharp\varphi\rangle_{T(\Gamma)}$ for all $\omega,\varphi\in T^*(\Gamma)$, so that
\begin{equation}\label{2.1}
  (\omega,\varphi)_\Gamma=(\sharp \varphi, \sharp\omega)_\Gamma,\qquad \text{for all $\omega,\varphi\in T^*(\Gamma)$}.
\end{equation}
By $|\cdot|_\Gamma^2= (\cdot,\cdot)_\Gamma$  we shall denote the associated bundle norms on $T(\Gamma)$ and $T^*(\Gamma)$.

Denoting by $d_\Gamma$ the standard differential on $\Gamma$, the Riemannian gradient operator $\nabla_\Gamma$ is defined  by
setting, for $u\in C^1(\Gamma)$ and thus by density for $u\in H^1(\Gamma)$,
$\nabla_\Gamma u=\sharp d_\Gamma u$,
so $\nabla_\Gamma u=g^{ij}\partial_ju\partial_i$
 in local coordinates, where $(g^{ij})=(g_{ij})^{-1}$. By \eqref{2.1} one trivially gets that
$(\nabla_\Gamma u,\nabla_\Gamma v)_\Gamma=(d_\Gamma u,d_\Gamma v)_\Gamma$ for all $u,v\in H^1(\Gamma)$, so in the sequel the use of vectors or forms is optional.

It is well known, see for example \cite[Chapter 3]{mugnvit}, that $H^1(\Gamma)$ can be equipped with the equivalent norm $\|\cdot\|_{H^1(\Gamma)}$ given by
\begin{equation}\label{2.2}
\|u\|_{H^1(\Gamma)}^2=\|u\|_{2,\Gamma}^2+\|\nabla_\Gamma u\|_{2,\Gamma}^2,\quad\text{where}\quad \|\nabla_\Gamma u\|_{2,\Gamma}^2:=\int_\Gamma |\nabla_\Gamma u|_\Gamma^2.
\end{equation}
In the sequel we shall also deal with the closed subspace of $H^1(\Gamma)$
\begin{equation}\label{2.3}
H^1_{\Gamma_0}(\Gamma)=\{u\in H^1(\Gamma): u=0\quad\text{a.e. on $\Gamma_0$}\}
\end{equation}
endowed with the  norm $\|\cdot\|_{H^1(\Gamma)}$, which is then a Hilbert space. Since
$\nabla_\Gamma u=0$ a.e. on $\Gamma\setminus\overline{\Gamma_1}$  for all $u\in H^1_{\Gamma_0}(\Gamma)$,   and  since $\mathcal{H}^{N-1}(\overline{\Gamma_0}\cap\overline{\Gamma_1})=0$, we have
\begin{equation}\label{2.4}
\|u\|_{H^1(\Gamma)}^2=\|u\|_{2,\Gamma_1}^2+\|\nabla_\Gamma u\|_{2,\Gamma_1}^2\qquad\text{for all $u\in H^1_{\Gamma_0}(\Gamma)$,}
\end{equation}
where $\|\nabla_\Gamma u\|_{2,\Gamma_1}^2:=\int_{\Gamma_1} |\nabla_\Gamma u|_\Gamma^2$.

\begin{rem}\label{Remark 2} Although the definition of the space $H^1_{\Gamma_0}(\Gamma)$ given above is adequate for our purpose, we would like to point out two characterizations of it in two different geometrical situations.
\renewcommand{\labelenumi}{{\roman{enumi})}}
\begin{enumerate}
\item When $\overline{\Gamma_0}\cap\overline{\Gamma_1}=\emptyset$  both $\Gamma_0$ and $\Gamma_1$ are relatively open. Hence, by identifying the elements of $H^1(\Gamma_i)$, $i=0,1$, with their trivial extensions to $\Gamma$, one easily gets the decomposition $H^1(\Gamma)=H^1(\Gamma_0)\oplus H^1(\Gamma_1)$. Consequently  $H^1_{\Gamma_0}(\Gamma)$ can be isometrically identified with $H^1(\Gamma_1)$, as one usually does.
\item  When $\overline{\Gamma_0}\cap\overline{\Gamma_1}\not=\emptyset$ such a characterization fails to hold.
       To show this fact, we claim that the characteristic function $\chi_{\Gamma_1}$ of $\Gamma_1$, defined on $\Gamma$ and considered (through its equivalence class) as an element of $L^2(\Gamma)$, does not belong to $H^1(\Gamma)$.
    Indeed, suppose by contradiction that $\chi_{\Gamma_1}\in H^1(\Gamma)$. Trivially $\nabla_\Gamma \chi_{\Gamma_1}=0$ a.e. on $\Gamma_1$ and, as proved above, $\nabla_\Gamma u=0$ a.e. on $\Gamma\setminus\overline{\Gamma_1}$. Since $\mathcal{H}^{N-1}(\overline{\Gamma_0}\cap\overline{\Gamma_1})=0$, we thus have $\nabla_\Gamma \chi_{\Gamma_1}=0$ a.e. on $\Gamma$. Since $\chi_{\Gamma_1}\in L^N(\Gamma)$ we then get $\chi_{\Gamma_1}\in W^{1,N}(\Gamma)$.
    Since $\Gamma$ is $C^1$, there is a sequence  $(\varphi_n)_n$ in $C^1(\Gamma)$
    such that $\varphi_n\to \chi_{\Gamma_1}$ in $W^{1,N}(\Gamma)$. Since $N>N-1$, by Morrey's Theorem, we then get that $\varphi_n\to \chi_{\Gamma_1}$ uniformly in $\Gamma$, so $\chi_{\Gamma_1}\in C(\Gamma)$, the desired contradiction.
    As a consequence, $\chi_{\Gamma_1}\not\in H^1_{\Gamma_0}(\Gamma)$.
    Trivially  the restriction of  $\chi_{\Gamma_1}$ to $\Gamma_1$  belongs to $H^1(\Gamma_1)$.

    In this case the elements of $H^1_{\Gamma_0}(\Gamma)$ ''vanish'' at the relative boundary $\partial\Gamma_1=\overline{\Gamma_0}\cap\overline{\Gamma_1}$ of $\Gamma_1$ on $\Gamma$, although such a notion can be made more precise only when $\partial\Gamma_1$ is regular enough. For example, when $\Gamma$ is smooth and $\overline{\Gamma_1}$ is a manifold with boundary $\partial\Gamma_1$,   see \cite[Chapter~4, \S 5, formula (5.1), p. 290]{taylor}, $H^1_{\Gamma_0}(\Gamma)$ is isometrically isomorphic to  the space
        $$H^1_0(\Gamma_1):=\overline{C^\infty_c(\Gamma_1)}^{\|\cdot\|_{H^1(\Gamma_1)}}.$$
\end{enumerate}
\end{rem}
The Laplace--Beltrami operator $\Delta_\Gamma$ can be defined in a geometrically elegant way by using $\nabla_\Gamma$ and the Riemannian divergence operator, as in \cite[\S~2.3]{mugnvit}, at least when $\Gamma$ is $C^2$. To avoid the necessity of introducing Sobolev spaces of tensor fields we shall adopt here a less elegant approach.
Indeed we set, when $\Gamma$ is $C^2$ and $u\in C^2(\Gamma')$, $\Gamma'\subset\Gamma$ relatively open,
\begin{equation}\label{2.5}
\Delta_{\Gamma} u=g^{-1/2}\partial_i( g^{1/2}g^{ij}\partial_j u), \quad\text{where $g=\det (g_{ij})$,}
\end{equation}
in local coordinates. Since $g$, $g^{ij}$ are continous and $\Gamma$ is compact, formula \eqref{2.5} extends by density to $u\in H^2(\Gamma)$, so defining an operator $-\Delta_\Gamma\in\mathcal{L}(H^2(\Gamma);L^2(\Gamma))$, which restricts to $-\Delta_\Gamma\in\mathcal{L}(H^2(\Gamma');L^2(\Gamma'))$ for relatively open subsets $\Gamma'$ of $\Gamma$. Since $\Gamma$ is compact, by \eqref{2.5}, integrating by parts and using a $C^2$ partition of the unity,  we get
\begin{equation}\label{2.6}
-\int_\Gamma \Delta_\Gamma u v=\int_\Gamma (\nabla_\Gamma u,\nabla_\Gamma v)_\Gamma\quad\text{for all $u\in H^2(\Gamma)$ and $v\in H^1(\Gamma)$.}
\end{equation}
Formula \eqref{2.6} motivates the definition of the operator $-\Delta_\Gamma\in\mathcal{L}(H^1(\Gamma);H^{-1}(\Gamma))$, also when $\Gamma$ is merely $C^1$, given by
\begin{equation}\label{2.7}
\langle -\Delta_\Gamma u,v\rangle _{H^1(\Gamma)}=\int_\Gamma (\nabla_\Gamma u, \nabla_\Gamma v)_\Gamma\quad\text{for all $u,v\in H^1(\Gamma)$.}
\end{equation}
By density, when $\Gamma$ is $C^2$, the so defined operator is the unique  extension of $-\Delta_\Gamma\in\mathcal{L}(H^2(\Gamma);L^2(\Gamma))$.

In \S~\ref{section 4} we shall deal with the realization of $-\Delta_\Gamma$ between the space $H^1_{\Gamma_0}(\Gamma)$ and its dual.
To motivate its definition we are now  going to briefly consider the two different cases pointed out in Remark~\ref{Remark 2}.

When $\overline{\Gamma_0}\cap\overline{\Gamma_1}=\emptyset$, so $\Gamma_1$ is compact, and $\Gamma$ is $C^2$, formula \eqref{2.6} holds  when $\Gamma$ is replaced by  $\Gamma_1$. In this case it is then natural to set  the operator $-\Delta_{\Gamma_1}\in\mathcal{L}(H^1(\Gamma_1);H^{-1}(\Gamma_1))$ like the operator $-\Delta_\Gamma$ above, that is by
\begin{equation}\label{2.8}
\langle -\Delta_{\Gamma_1} u,v\rangle _{H^1(\Gamma_1)}=\int_{\Gamma_1} (\nabla_\Gamma u, \nabla_\Gamma v)_\Gamma\quad\text{for all $u,v\in H^1(\Gamma_1)$.}
\end{equation}
When $\overline{\Gamma_0}\cap\overline{\Gamma_1}\not=\emptyset$, $\Gamma$ is smooth and $\overline{\Gamma_1}$ is a manifold with boundary $\partial\Gamma_1$, formula \eqref{2.6} does not hold anymore on $\Gamma_1$, since a boundary integral on $\partial\Gamma_1$ appears.
On the other hand, taking into account the homogeneous Dirichlet boundary condition in the space $H^1_0(\Gamma_1)$, it is natural to set $-\Delta^D_{\Gamma_1}\in\mathcal{L}(H^1_0(\Gamma_1);H^{-1}(\Gamma_1))$ by
\begin{equation}\label{2.9}
\langle -\Delta_{\Gamma_1}^D u,v\rangle _{H^1_0(\Gamma_1)}=\int_{\Gamma_1} (\nabla_\Gamma u, \nabla_\Gamma v)_\Gamma\quad\text{for all $u,v\in H^1_0(\Gamma_1)$,}
\end{equation}
where as usual $H^{-1}(\Gamma_1):=[H^1_0(\Gamma_1)]'$.
To include, as particular cases, the two operators $-\Delta_{\Gamma_1}$ and $-\Delta^D_{\Gamma_1}$ given by \eqref{2.8} and \eqref{2.9}, in the sequel we shall deal with the operator $-\Delta^D_{\Gamma_1}\in\mathcal{L}(H^1_{\Gamma_0}(\Gamma);[H^1_{\Gamma_0}(\Gamma)]')$
defined by
\begin{equation}\label{2.10}
\langle -\Delta^D_{\Gamma_1} u,v\rangle _{H^1_{\Gamma_0}(\Gamma)}=\int_{\Gamma_1} (\nabla_\Gamma u, \nabla_\Gamma v)_\Gamma\quad\text{for all $u,v\in H^1_{\Gamma_0}(\Gamma_1)$,}
\end{equation}
noticing that, by \eqref{2.7}, $-\Delta^D_{\Gamma_1}u=-(\Delta_{\Gamma}u)_{|H^1_{\Gamma_0}(\Gamma)}$ for all $u\in H^1_{\Gamma_0}(\Gamma)$.
\subsection{The space $H^1$.} \label{section 2.3} We recall, see \cite[Lemma 1, p. 2147]{vazvitHLB},  trivially extending to $\Gamma$ of class $C^1$, that the space
$$H^1(\Omega;\Gamma)=\{(u,v)\in H^1(\Omega)\times H^1(\Gamma): v=u_{|\Gamma}\},$$
endowed with the topology inherited from the product, can be identified, through the bijective isomorphism $(u,u_{|\Gamma})\mapsto u$, with the space $\{u\in H^1(\Omega): u_{|\Gamma}\in H^1(\Gamma)\}$ and equivalently equipped with the norm $\|\cdot\|_{H^1(\Omega,\Gamma)}$ given by
$$\|u\|^2_{H^1(\Omega,\Gamma)}:=\|\nabla u\|_2^2+\|\nabla_\Gamma u\|_{2,\Gamma}^2+\|u\|_{2,\Gamma}^2.$$
The identification made in \S~\ref{section 1.2} between the spaces $H^1$ and $H^1_{\Gamma_0}(\Omega,\Gamma)$, respectively defined  by \eqref{1.6} and \eqref{1.17}, is a simple consequence of the identification above.

By using formula \eqref{2.4}, we can equip  $H^1$ with the norm $\vertiii\cdot_{H^1}$ given by
\begin{equation}\label{2.11}
 \vertiii u_{H^1}^2= \|\nabla u\|_2^2+\|\nabla_\Gamma u\|_{2,\Gamma_1}^2+\|u\|_{2,\Gamma_1}^2.
\end{equation}
On the other hand, to get advantage of the connectedness of $\Omega$ and of the assumption $\mathcal{H}^{N-1}(\Gamma_0)>0$, made in the present paper, we point out the following well--known result, referring to \cite{MP}--\cite{MPCorrigendum} for a proof.
 \begin{lem}\label{Lemma 1} Let $\Omega$ be connected and $\mathcal{H}^{N-1}(\Gamma_0)>0$. Then, setting, for $u,v\in H^1$,
\begin{equation}\label{2.12}
  (u,v)_{H^1}=\int_\Omega\nabla u\nabla v+\int_{\Gamma_1}(\nabla_\Gamma u,\nabla_\Gamma v)_\Gamma \quad\text{and}\quad \|\cdot\|_{H^1}=(\cdot,\cdot)_{H^1}^{1/2},
\end{equation}
 $\|\cdot\|_{H^1}$ defines on $H^1$ a norm equivalent to  $\vertiii\cdot_{H^1}$.
\end{lem}
\subsection{Some results from Critical Point Theory} \label{section 2.4}
We now recall some well--known notions of Critical Point Theory (referring to \cite{ambrosettimalchiodi}) for a functional $\mathcal{I}\in C^1(E)=C^1(E;\R)$ on any Banach space $E$ with norm $\|\cdot\|_E$. By $d\mathcal{I}\in C(E;E')$ we shall denote the Fr\'{e}chet differential of $\mathcal{I}$. Moreover, when $E$ is a Hilbert space endowed with the scalar product $(\cdot,\cdot)_E$, by $\nabla\mathcal{I}\in C(E;E)$ we shall denote the gradient of $\mathcal{I}$, which is defined as follows (see also
\cite[Chapter~5,~\S~5.2,~p.78]{ambrosettimalchiodi}). By the Riesz Theorem, for any $u\in E$ one sets $\nabla\mathcal{I}(u)=w$, where $w\in E$ is the unique solution (in $E$) of the equation
\begin{equation}\label{2.13}
(w,v)_E=\langle d\mathcal{I}(u),v\rangle_E\qquad\text{for all $v\in E$.}
\end{equation}
Clearly, by the Riesz Theorem, we also have $\|\nabla \mathcal{I}(u)\|_E=\|d\mathcal{I}(u)\|_{E'}$.
In the sequel we shall also use the following terminology.
 \begin{definition}\label{Definition 0}
Let $\mathcal{I}\in C^1(E)$. We say that a sequence $(u_n)_n$ in $E$  is a
Palais--Smale (in short, (PS)) sequence if $(\mathcal{I}(u_n))_n$ is bounded and $d\mathcal{I}(u_n)\to 0$ in $E'$. We also say that $\mathcal{I}\in C^1(E)$ satisfies the (PS) condition if any (PS) sequence has a (strongly) convergent subsequence.
\end{definition}
\begin{rem}\label{Remark 3}Clearly, when $E$ is a Hilbert space, as seen above, $d\mathcal{I}(u_n)\to 0$ in $E'$ if and only if
$\nabla\mathcal{I}(u_n)\to 0$ in $E$.
\end{rem}
The following result is nothing but a well--known version of  the celebrated Mountain Pass Theorem, see \cite[Chapter 1, p. 4]{rabinowitz}.
\begin{thm}[\bf Mountain Pass Theorem, standard version]\label{Theorem 7} Let $\mathcal{I}\in C^1(E)$ satisfies the (PS) condition and
\renewcommand{\labelenumi}{{\roman{enumi})}}
\begin{enumerate}
\item $\mathcal{I}(0)=0$;
\item there are $\rho,\eta>0$ such that $\mathcal{I}(u)\ge \eta$ for all $u\in E$ such that $\|u\|_X=\rho$;
\item there is $l\in E$ such that $\|l\|_E>\rho$ and $\mathcal{I}(l)\leq 0$.
\end{enumerate}
Then $\mathcal{I}$ possesses a critical point $u\in E$ such that $\mathcal{I}(u)=c_l\ge \eta$, where
$$c_l=\inf_{\sigma\in \Sigma_l}\max_{t\in [0,1]}\mathcal{I}(\sigma(t)), \text{where } \Sigma_l=\{\sigma\in C([0,1];E): \sigma (0)=0, \sigma(1)=l\}.
$$
\end{thm}
Textbooks in Critical Point Theory usually do not point out (since, when looking for critical points, this remark would be not of interest) that the critical level $c_l$ above may depend on $l$. An explicit elementary example is given in \cite[Example 1]{MP}.
Since in this  paper we are interested in characterizing our critical level as the potential--well depth of the functional $\mathcal{I}$, we are now going to recall  a less known variant of the Mountain Pass Theorem under slighty more restrictive assumptions on the functional. They look similar, although not identical, to the assumptions in the first version of the Mountain Pass Theorem, that is in  \cite[Theorem 2.1, p. 354]{AmbrosettiRabinowitz}.
They are pointed out, without further detail, in \cite[Chapter~8,~\S~8.1,Remark~8.3]{ambrosettimalchiodi}.
A complete proof of the following result, based on Therem~\ref{Theorem 7}, can be found in \cite[Proof~of~Theorem~5,~p.819]{MP}.
\begin{thm}[\bf Mountain Pass Theorem, variant]\label{Theorem 8} Let $\mathcal{I}\in C^1(I)$ satisfies the (PS) condition, assumptions i)--iii) in Theorem~\ref{Theorem 7} and
\renewcommand{\labelenumi}{{\roman{enumi})}}
\begin{itemize}
\item[iv)] $\mathcal{I}(u)>0$ for all $u\in E$ such that $0<\|u\|_E\le\rho$.
\end{itemize}
Then $\mathcal{I}$ possesses a critical point $u\in E$ such that $\mathcal{I}(u)=c$, where $c\ge \eta$ is given by
\begin{equation}\label{2.14}
c=\inf_{\sigma\in \Sigma}\max_{t\in [0,1]}\mathcal{I}(\sigma(t)), \text{and } \Sigma=\{\sigma\in C([0,1];E): \sigma (0)=0,\, \mathcal{I}(\sigma(1))<0\}.
\end{equation}
\end{thm}
In the sequel we shall also use the following $\Z_2$--version of the Mountain Pass Theorem, see \cite[Chapter 9, Theorem 9.12, p. 55  and Proposition 9.33, p. 58]{rabinowitz}.
\begin{thm}[\bf $\Z_2$--Mountain Pass Theorem]\label{Theorem 9} Let $E$ be  infinite dimensional and $\mathcal{I}\in C^1(E)$ be even, satisfying the (PS) condition, assumptions i)--ii) of Theorem~\ref{Theorem 7} and
\begin{itemize}
\item[v)] for each finite dimensional subspace $Y$ of $E$ there is $R_Y>0$ such that $\mathcal{I}(u)\le 0$ for all $u\in Y$ such that $\|u\|_E>R_Y$.
\end{itemize}
Then $\mathcal{I}$ possesses a sequence $(u_n)_n$ of critical points such that $\mathcal{I}(u_n)\to\infty$.
\end{thm}
\subsection{On the assumptions (A1--4)}\label{section 2.5} This subsection is devoted to give a further example  of a couple $(f,g)$ satisfying assumptions (A1--4), and also to point out  some further consequences of them.

In addition to the examples given in \eqref{1.10} and \eqref{1.11}, another couple $(f,g)$ of functions satisfying assumptions (A1--4) is the following one:
\begin{gather}\label{2.15}
\begin{aligned}
f(u)=
&\begin{cases}
\gamma_1|u-a_1|^{p_1-2}(u-a_1),\, &\text{for $u\le a_1$,}\\
\phantom{\gamma_1|u-a_1|}0,&\text{for $a_1\le u\le a_2$,}\\
\gamma_2|u-a_2|^{p_2-2}(u-a_2),\, &\text{for $u\ge a_2$,}
\end{cases}\\
g(u)=
&\begin{cases}
\delta_1|u-b_1|^{q_1-2}(u-b_1),\, &\text{for $u\le b_1$,}\\
\phantom{\gamma_1|u-b_1|}0,&\text{for $b_1\le u\le b_2$,}\\
\delta_2|u-b_2|^{q_2-2}(u-b_2),\, &\text{for $u\ge b_2$,}
\end{cases}
\end{aligned}\\
\label{2.16}
\begin{gathered}
\text{where}\qquad  a_1,b_1<0<a_2,b_2,\qquad \gamma_1 \gamma_2+\delta_1\delta_2>0,\\
2<p_1,\,p_2<\romega, \quad 2<q_1,\,q_2<\rgamma, \quad \gamma_1,\,\gamma_2,\,\delta_1,\,\delta_2\ge 0.
\end{gathered}
\end{gather}
Indeed, assumptions (A1--2) trivially hold. Moreover, since (for example) just taking the derivative one checks that  the function $u\mapsto (u-a)^{p-1}/u$ is, when $p>2$, increasing for $u\ge a\ge 0$, one easily gets that assumption (A2) holds true, with $\ell^\pm=\mathfrak{m}^\pm=0$. Hence $h\equiv f$ and $k\equiv g$ in \eqref{1.7}. Also assumption (A3) holds true since, choosing  $p_0=\min\{p_1,p_2\}$, and considering for example $u\ge a_2>0$, one has $F(u)=\gamma_2(u-a_2)^{p_2}/p_2$ and then
\begin{align*}
h(u)u-p_0H(u)=& f(u)u-p_0F(u)=\gamma_2 (u-a_1)^{p_2-1}\left[u-\frac{p_0}{p_2}(u-a_2)\right]\\
=&\gamma_2 (u-a_1)^{p_2-1}[(p_2-p_0)u+p_0a_2]\ge 0.
\end{align*}
Finally, $\sigma$ and $\xi$ have no critical points where they do not vanish and, since $\gamma_1,\gamma_2>0$ or $\delta_1,\delta_2>0$, assumption \eqref{1.9} holds.

\begin{rem}\label{Remark 4} Since $\ell^\pm,\mathfrak{m}^\pm>-\infty$, assumption (A2) implies that $f(0)=\lim_{u\to 0} f(u)=0$ and
$g(0)=\lim_{u\to 0} g(u)=0$. Since, by (A1), $\lim_{u\to 0^\pm} f'(u)$ and  $\lim_{u\to 0^\pm} g'(u)$ exist, by de l'H\^{o}pital rule we have $\ell^\pm=\lim_{u\to 0^\pm} f'(u)$ and  $\mathfrak{m}^\pm=\lim_{u\to 0^\pm} g'(u)$. Consequently, by \eqref{1.7}, we have
\begin{equation}\label{2.17}
h,k\in C^1(\R),\qquad\text{and}\quad h'(0)=k'(0)=h(0)=k(0)=0.
\end{equation}
Next, assumption (A1) also yields that $|h'(u)|=O(|u|^{p-2})$ and $|k'(u)|=O(|u|^{q-2})$ as $|u|\to\infty$. Combining it with  \eqref{2.17} we also get the existence of nonnegative constants $c_1=c_1(f,p)$ and $c_2=c_2(g,q)$ such that
\begin{equation}\label{2.18}
\left\{
\begin{alignedat}2
&|h'(u)|\le c_1(1+|u|^{p-2}), &&\qquad |h(u)|\le c_1(1+|u|^{p-1}), \\
& |f(u)|\le c_1(1+|u|^{p-1}), && \qquad H(u)\le c_1(1+|u|^p),\\
&|k'(u)|\le c_2(1+|u|^{q-2}), &&\qquad |k(u)|\le c_2(1+|u|^{q-1}), \\
& |g(u)|\le c_2(1+|u|^{q-1}), && \qquad K(u)\le c_2(1+|u|^q),
\end{alignedat}\right.
\quad \text{for all $u\in\R$.}
\end{equation}
Moreover, recalling the functions $\sigma,\xi\in C(\R)$ introduced in Remark~\ref{Remark0}, since $\sigma,\xi\ge 0$ in $\R$, we get
\begin{equation}\label{2.19}
h(u)u\ge 0,\quad k(u)u\ge 0,\quad H(u)\ge 0,\quad \text{and}\quad K(u)\ge 0\quad\text{for all $u\in\R$.}
\end{equation}
Next, by \eqref{2.18}, we also get the existence of nonnegative constants $c_3=c_3(f,p)$ and $c_4=c_4(g,q)$ such that
\begin{equation}\label{2.20}
0\le \sigma(u)\le c_3(1+|u|^{p-2})\quad\text{and}\quad 0\le \xi(u)\le c_4(1+|u|^{q-2})\quad\text{for all $u\in\R$.}
\end{equation}
Moreover, one trivially has $\sigma, \xi\in C^1(\R\setminus\{0\})$. Since by \eqref{1.7}  we  have
\begin{equation}\label{2.21}
h'(u)=\begin{cases}
f'(u)-\ell^+,\,\,&\text{if $u>0$,}\\
f'(u)-\ell^-,\,\,&\text{if $u<0$,}
\end{cases}\quad
k'(u)=\begin{cases}
g'(u)-\mathfrak{m}^+,\,\,&\text{if $u>0$,}\\
g'(u)-\mathfrak{m}^-,\,\,&\text{if $u<0$,}
\end{cases}
\end{equation}
by also using  \eqref{1.12} we get that
\begin{equation}\label{2.21bis}
h'(u)=\sigma'(u)u+\sigma(u), \quad \text{and}\quad k'(u)=\xi'(u)u+\xi(u)\quad\text{ for all $u\not=0$.}
\end{equation}
Hence, by \eqref{2.17}, by setting the functions $\sigma'(u)u$ and $\xi'(u)u$ to vanish at $0$, they continuously extend to the whole of $\R$. In the sequel we shall always deal with these extensions, without further comment. Hence, by \eqref{2.18} and \eqref{2.20},
there are nonnegative constants $c_5=c_5(f,p)$ and $c_6=c_6(g,q)$ such that
\begin{equation}\label{2.22}
0\le \sigma'(u)u\le c_5(1+|u|^{p-2}),\quad 0\le \xi'(u)u\le c_6(1+|u|^{q-2})\quad\text{for all $u\in\R$.}
\end{equation}
\end{rem}
The following result points out other consequences of assumptions (A1--4).
\begin{lem}\label{Lemma 2} Let assumptions (A1--4) hold. Then there are nonnegative constants
$C_f^\pm$, $C_g^\pm$, $c_7,c_8$ and $a_0<\lambda_1$,
depending on $f,g,p,q,p_0$ and $q_0$, such that
\begin{alignat}3
& F(u)\ge -c_7u^2+C_f^\pm |u|^{p_0}, && \quad G(u)\ge -c_7u^2+C_g^\pm |u|^{q_0} &&\quad\text{for $\pm u\ge 0$;}\label{2.23}\\
& \tfrac 12 f(u)u\ge -c_7u^2+C_f^\pm |u|^{p_0}, && \quad \tfrac 12 g(u)u\ge -c_7u^2+C_g^\pm |u|^{q_0} &&\quad\text{for $\pm u\ge 0$;}\label{2.24}\\
& F(u)\le \tfrac 12 a_0 u^2+c_8 |u|^p, && \quad G(u)\le \tfrac 12 a_0u^2+c_8|u|^q &&\quad\text{for all  $u\in\R$;}\label{2.25}\\
& f(u)u\le a_0u^2+c_8 |u|^p, &&\quad  g(u)u\le a_0u^2+c_8 |u|^q &&\quad\text{for $u\in\R$.}\label{2.26}
\end{alignat}
Moreover, the following implications hold true:
\begin{align}\label{2.27}
\lim_{u\to\infty}\sigma(u)>0\Longrightarrow C_f^+>0,
&\qquad \lim_{u\to-\infty}\sigma(u)>0\Longrightarrow C_f^->0,\\
\label{2.28}\lim_{u\to\infty}\xi(u)>0\Longrightarrow C_g^+>0,
&\qquad \lim_{u\to-\infty}\xi(u)>0\Longrightarrow C_g^->0.
\end{align}
\end{lem}
\begin{proof}
We first notice that, by \eqref{1.7} and \eqref{1.14}, we have
\begin{equation}\label{2.29}
\left\{\,\,
\begin{aligned}
&f(u)u=\ell^+(u^+)^2+\ell^-(u^-)^2+h(u)u, \\
&  g(u)u=\mathfrak{m}^+(u^+)^2+\mathfrak{m}^-(u^-)^2+k(u)u,\\
&F(u)=\tfrac 12\left[\ell^+(u^+)^2+\ell^-(u^-)^2\right]+H(u), \\
&  G(u)=\tfrac 12 \left[\mathfrak{m}^+(u^+)^2+\mathfrak{m}^-(u^-)^2\right]+K(u),
\end{aligned}\quad\text{for all $u\in\R$.}
\right.
\end{equation}
Hence, by setting
\begin{equation}\label{2.30}
a_m=\min\{\ell^\pm,\mathfrak{m}^\pm\},\qquad \text{and}\quad a_M=\max\{\ell^\pm,\mathfrak{m}^\pm\},
\end{equation}
by \eqref{1.6} we have $a_m\le a_M<\lambda_1$ and the following preliminary  estimates hold true:
\begin{equation}\label{2.31}
\left\{
\,\,
\begin{alignedat}3
&a_mu^2+h(u)u &&\le f(u)u &&\le a_Mu^2+h(u)u,\\
&a_mu^2+k(u)u &&\le g(u)u &&\le a_Mu^2+k(u)u,\\
&\tfrac 12 a_mu^2+H(u) &&\le F(u) &&\le \tfrac 12 a_Mu^2+H(u),\\
&\tfrac 12 a_mu^2+K(u) &&\le G(u) &&\le \tfrac 12 a_Mu^2+K(u),
\end{alignedat}
\right.\qquad\text{for all $u\in\R$.}
\end{equation}
To prove \eqref{2.23} and \eqref{2.24} we are at first going to consider the behavior of $\sigma$ in $[0,\infty)$.
We shall distinguish between the two cases i) and ii) outlined in Remark~\ref{Remark0}. When $\lim_{u\to\infty}\sigma(u)>0$, by the monotonicity of $\sigma$ we can take the parameter $M$ in assumption (A3) so large that $\sigma(u)>0$ when $u\ge M/2$, so $h(u)u>0$
when $u\ge M/2$. Then, by \eqref{1.7} and \eqref{2.19}, we have $H(u)>0$ for $u\ge M$. By integrating the differential inequality
$H'(u)u\ge p_0H(u)$ in \eqref{1.8}, we then get that
\begin{equation}\label{2.32}
\frac{h(u)u}{p_0}\ge H(u)\ge \frac{H(M)|u|^{p_0}}{M^p_0}\quad\text{for $u\ge M$.}
\end{equation}
Since $p_0>2$, by \eqref{2.31} and \eqref{2.32} we get
\begin{equation}\label{2.33}
\min\left\{\tfrac 12 f(u)u,F(u)\right\}\ge \tfrac 12 a_mu^2+\frac{H(M)|u|^{p_0}}{M^p_0}\quad\text{for $u\ge M$.}
\end{equation}
By \eqref{2.19} and \eqref{2.31}, when $0\le u\le M$ we have
\begin{equation}\label{2.34}
\begin{aligned}
\min\left\{\tfrac 12 f(u)u,F(u)\right\}& \ge \tfrac 12 a_mu^2=\tfrac 12 (a_m-1)u^2+\tfrac 12 |u|^{2-p_0}|u|^{p_0}\\
&\ge \tfrac 12 (a_m-1)u^2+\tfrac 12 M^{2-p_0}|u|^{p_0}.
\end{aligned}
\end{equation}
Combining \eqref{2.33} with  \eqref{2.34}, when $\lim_{u\to\infty}\sigma(u)>0$ we then get
\begin{equation}\label{2.35}
\min\left\{\tfrac 12 f(u)u,F(u)\right\}\ge \tfrac 12 (a_m-1)u^2+C_f^+|u|^{p_0}\quad\text{for $u\ge 0$,}
\end{equation}
where $C_f^+:=\min\left\{\frac{H(M)}{M^{p_0}},\frac{M^{2-p_0}}2\right\}>0$.

On the other hand, when $\lim_{u\to\infty}\sigma(u)=0$, since $H(u)\ge 0$ and $h(u)u\ge 0$, by \eqref{2.31} we get that \eqref{2.35} continues to hold, provided in this case we set $C_f^+=0$.
In conclusion, \eqref{2.35} holds in both cases, with $C_f^+>0$ in the first one.

A simple repetition of the arguments used to prove \eqref{2.35} shows that
\begin{align*}
&\min\left\{\tfrac 12 f(u)u,F(u)\right\}\ge \tfrac 12 (a_m-1)u^2+C_f^-|u|^{p_0}\quad\text{for $u\le 0$, and}\\
&\min\left\{\tfrac 12 g(u)u,G(u)\right\}\ge \tfrac 12 (a_m-1)u^2+C_g^\pm|u|^{q_0}\quad\text{for $\pm u\ge 0$,}
\end{align*}
where $C_f^->0$ when $\lim_{u\to-\infty}\sigma(u)>0$, $C_g^+>0$ when $\lim_{u\to\infty}\xi(u)>0$, and  $C_g^->0$ when $\lim_{u\to-\infty}\xi(u)>0$. The estimates \eqref{2.23} and \eqref{2.24} are thus proved, together with the implications \eqref{2.27} and \eqref{2.28}.

To prove the estimates \eqref{2.25} and \eqref{2.26}, we are first going to consider the behavior of $f$ in $[0,\infty)$. Since $f(u)/u\to \ell^+$ as $u\to 0^+$, for any $\varepsilon>0$ there is $\rho_0>0$ such that $f(u)\le (\ell^++\varepsilon)u$ for $0\le 0<\rho_0$, so that $f(u)u\le (\ell^++\varepsilon)u^2$ and $F(u)\le \tfrac 12 (\ell^++\varepsilon)u^2$ for $0\le 0<\rho_0$.
Now, by \eqref{2.18}, $|f(u)u|=O(|u|^p)$ and $F(u)=O(|u|^p)$ when $|u|\to\infty$. Consequently, the functions
$$u\mapsto \frac{f(u)u-(\ell^++\varepsilon)u^2}{|u|^p}\qquad\text{and}\quad u\mapsto \frac{F(u)-\tfrac 12 (\ell^++\varepsilon)u^2}{|u|^p}$$
are bounded in $[\rho_0,\infty)$. Then there is a nonnegative constant $c_9(\varepsilon)$, also depending on $f$ and $p$, such that
$$f(u)u\le (\ell^++\varepsilon)u^2+c_9(\varepsilon)|u|^p,\quad\text{and}\quad F(u)\le \tfrac 12(\ell^++\varepsilon)u^2+c_9(\varepsilon)|u|^p\quad\text{for all $u\ge 0$.}$$
By simply repeating the same arguments when dealing with the behavior of $f$ in $(-\infty,0]$ and those of $g$ in $[0,\infty)$ and $(-\infty,0]$, and recalling \eqref{2.30}, we obtain the existence of a nonnegative constant $c_{10}(\varepsilon)$, also depending on $f,g,p$ and $q$, such that
\begin{equation}\label{2.36}
\begin{alignedat}2
& f(u)u\le (a_M+\varepsilon)u^2+c_{10}(\varepsilon)|u|^p, &&\quad F(u)\le \tfrac 12 (a_M+\varepsilon)u^2+c_{10}(\varepsilon)|u|^p,\\
& g(u)u\le (a_M+\varepsilon)u^2+c_{10}(\varepsilon)|u|^q, &&\quad G(u)\le \tfrac 12 (a_M+\varepsilon)u^2+c_{10}(\varepsilon)|u|^q
\end{alignedat}
\end{equation}
for all $u\in\R$. By choosing $\varepsilon=\varepsilon_0:=(\lambda_1-a_M)/2>0$  and $c_8=:=c_{10}(\varepsilon_0)$, by \eqref{2.36} we get the estimates \eqref{2.25} and \eqref{2.26}, with
\begin{equation}\label{2.37}
a_0:=a_M+\varepsilon_0=\frac {\lambda_1+a_M}2<\lambda_1,
\end{equation}
completing the proof.
\end{proof}
\section{Mountain Pass type solutions of \eqref{1.1}} \label{section 3}
\subsection{Weak solutions} \label{section 3.1} We start by recalling what we mean by a weak solution of \eqref{1.2}, referring to \cite[\S 2.2 and Definition~3.1, p. 4896]{Dresda2}.
\begin{definition}\label{Definition 1} Let  (A1--4)  and \eqref{1.13} hold.  A \emph{weak solution} of problem \eqref{1.2} is
\begin{equation}\label{3.1}
  u\in L^\infty_\loc([0,\infty);H^1)\cap W^{1,\infty}_\loc ([0,\infty);H^0)
\end{equation}
such that $\alpha^{1/m}u_t\in L^m_\loc([0,\infty);L^m(\Omega))$, $\beta^{1/\mu}(u_{|\Gamma})_t\in L^\mu_\loc([0,\infty);L^\mu(\Gamma_1))$
\begin{footnote}{in the sequel we shall write $(u_{|\Gamma})_t$  and $(\psi_{|\Gamma})_t$ simply as $u_
t$ and $\psi_t$, for the sake of simplicity, recalling that the time derivative of them at the boundary is always taken in this sense.}
\end{footnote}
and, for all $\psi\in C_c((0,\infty);H^1)\cap C^1_c((0,\infty); H^0)$ such that $\alpha^{1/m}\psi_t\in L^m_\loc([0,\infty);L^m(\Omega))$, $\beta^{1/\mu}\psi_t\in L^\mu_\loc([0,\infty);L^\mu(\Gamma_1))$, the distribution identity
\begin{multline}\label{3.2}
  \int_0^\infty\left[-\int_\Omega u_t\psi_t-\int_{\Gamma_1}u_t \psi_t
   +\int_\Omega \nabla u \nabla\psi +\int_{\Gamma_1}(\nabla_\Gamma u,\nabla_\Gamma\psi)_\Gamma\right.\\
   \left. +\int_\Omega \alpha |u_t|^{m-2}u_t\psi+\int_{\Gamma_1}\beta|u_t|^{\mu-2}u_t\psi-\int_\Omega f(u)\psi-\int_{\Gamma_1}g(u)\psi \right]=0,
\end{multline}
holds.
We say that $u$ is \emph{stationary} if  $u(t)\equiv u_0\in H^1$  for all $t\ge 0$.
\end{definition}
We also make precise what we mean by weak solutions of \eqref{1.1}.
\begin{definition}\label{Definition 2} Let assumptions (A1--4) hold.  A \emph{weak solution} of problem \eqref{1.1} is $u\in H^1$ such that
\begin{equation}\label{3.3}
  \int_\Omega \nabla u\nabla \phi+\int_{\Gamma_1}(\nabla_\Gamma u,\nabla_\Gamma \phi)_\Gamma-\int_\Omega f(u)\phi-\int_{\Gamma_1}g(u)\phi=0\quad\text{for all $\phi\in H^1$.}
\end{equation}
\end{definition}
Actually weak solutions of \eqref{1.1} and stationary weak solutions of \eqref{1.2} coincide when they are both defined, as the following result shows.
\begin{lem}\label{Lemma 3}
Let  assumptions (A1--4) and \eqref{1.13} hold.  Then $u\equiv u_0\in H^1$ is a stationary weak  solution of \eqref{1.2} if and only if $u_0$ is a weak solution of \eqref{1.1}.
\end{lem}
\begin{proof}If $u_0$ is a weak solution of \eqref{1.1}, by \eqref{3.3}, one immediately gets that $u\equiv u_0$ satisfies \eqref{3.2}, so it is a weak stationary solution  of \eqref{1.2}. To prove the converse we recall that, by \cite[Lemma 3.3, p. 4896]{Dresda2}, any weak solution $u$ of \eqref{1.2} satisfies, for all $T>0$ and $\psi\in C([0,T];H^1)\cap C^1([0,T]; H^0)$ such that $\alpha^{1/m}\psi_t\in L^m_\loc([0,\infty); L^m(\Omega))$, $\beta^{1/\mu}\psi_t\in L^\mu_\loc([0,\infty);L^\mu(\Gamma_1))$, the distribution identity
\begin{multline}\label{3.4}
  \left[\int_\Omega u_t\psi+\int_{\Gamma_1}u_t\psi\right]_0^T+\int_0^T\left[-\int_\Omega u_t\psi_t-\int_{\Gamma_1}u_t\psi_t +\int_\Omega \nabla u \nabla\psi\right.\\
  \left.  +\int_{\Gamma_1}\negquad(\nabla_\Gamma u,\nabla_\Gamma\psi)_\Gamma+\int_\Omega\negquad \alpha |u_t|^{m-2}u_t\psi+\int_{\Gamma_1}\negquad\beta |u_t|^{\mu-2}u_t\psi -\int_\Omega\negquad f(u)\psi-\int_{\Gamma_1}\negquad g(u)\psi \right]=0.
\end{multline}
Hence, when $u\equiv u_0\in H^1$ is a stationary weak solution of \eqref{1.2}, taking in \eqref{3.4} test functions $\psi\equiv \phi\in H^1$, for an arbitrary $T>0$, we get \eqref{3.3}.
\end{proof}
The following result shows that equation \eqref{3.3}  has a variational structure and gives some properties which will be used in the sequel.
\begin{lem}\label{Lemma 4}  Let assumptions (A1--4) hold, and let the functionals $I,J:H^1\to\R$ be respectively defined by \eqref{1.18} and
\begin{equation}\label{3.5}
 J(u)=\int_\Omega F(u)+\int_{\Gamma_1} G(u).
\end{equation}
Them $I,J\in C^1(H^1)$ and $dJ\in C(H^1;(H^1)')$ is compact. Moreover, critical points of $I$ coincide with weak solutions of problem \eqref{1.1}.
\end{lem}
\begin{proof}By classical arguments, see \cite[Chapter 1, Theorem~2.9, p. 22]{ambrosettiprodi} and
\cite[Chapter 1, Theorem~1.8, p. 7]{ambrosettimalchiodi}, the potential operator $\mathcal{F}:H^1(\Omega)\to\R$, defined by
$\mathcal{F}(u)=\int_\Omega F(u)$, is Fr\'{e}chet differentiable, one has
\begin{equation}\label{3.6}
  \langle dF(u),\phi\rangle_{H^1(\Omega)}=\int_\Omega f(u)\phi\qquad\text{for all $u,\phi\in H^1(\Omega)$,}
\end{equation}
and, finally, $d\mathcal{F}:H^1(\Omega)\to [H^1(\Omega)]'$ is continuous and compact.
The same arguments show that the potential operator $\mathcal{G}:H^1(\Gamma_1)\to\R$, defined by
$\mathcal{G}(v)=\int_{\Gamma_1} G(v)$, is Fr\'{e}chet differentiable, one has
\begin{equation}\label{3.7}
  \langle d\mathcal{G}(v),\psi\rangle_{H^1(\Gamma_1)}=\int_{\Gamma_1}g(v)\psi\qquad\text{for all $v,\psi\in H^1(\Gamma_1)$,}
\end{equation}
and $d\mathcal{G}:H^1(\Gamma_1)\to [H^1(\Gamma_1)]'$ is continuous and compact as well.
Since $H^1\hookrightarrow H^1(\Omega)$, $\Tr\in\mathcal{L}(H^1; H^1(\Gamma_1))$, and $J(u)=\mathcal{F}(u)+\mathcal{G}(\Tr u)$, using \eqref{3.6} and \eqref{3.7}, we get that $J$ is Fr\'{e}chet differentiable, one has
\begin{equation}\label{3.8}
  \langle dJ(u),\phi\rangle_{H^1}=\int_\Omega f(u)\phi+\int_{\Gamma_1}g(u)\phi\qquad\text{for all $u,\phi\in H^1$,}
\end{equation}
and $dJ:H^1\to [H^1]'$ is continuous and compact as well.
Since, by \eqref{1.18}, \eqref{2.12} and \eqref{3.5} we have $I(u)=\tfrac 12\|u\|_{H^1}^2-J(u)$, we get that
$I\in C^1(H^1)$, with
\begin{equation}\label{3.9}
\begin{aligned}
\langle dI(u),\phi\rangle_{H^1}&=(u,\phi)_{H^1}-\langle dJ(u),\phi\rangle_{H^1}\\
&
=\int_\Omega \nabla u\nabla\phi+\int_{\Gamma_1}(\nabla_\Gamma u,\nabla_\Gamma \phi)_\Gamma-\int_{\Omega}f(u)\phi-\int_{\Gamma_1}g(u)\phi
\end{aligned}
\end{equation}
for all $u,\phi\in H^1$, and we also get that $dI:H^1\to(H^1)'$ is compact.
By comparing  \eqref{3.3} with  \eqref{3.9} one immediately  gets that \eqref{3.3} can be rewritten
as $dI(u)=0$, concluding the proof.
\end{proof}
We are now going to check some  properties of $I$.
\begin{lem}\label{Lemma 5} Let  (A1--4) hold. Then  $I$ satisfies the assumptions i)-iii) of Theorem~\ref{Theorem 7} and   iv) of Theorem~\ref{Theorem 8}.
Moreover, when $\lim_{u\to\pm\infty} \sigma(u)>0$, it also satisfies the assumption v) of Theorem~\ref{Theorem 9}.
\end{lem}
\begin{proof}
By \eqref{1.18} we have $I(0)=0$, so assumption  i) of Theorem~\ref{Theorem 7} holds. To check the assumption ii) of the same Theorem and assumption iv) of Theorem~\ref{Theorem 8} we are going to estimate $I$ from below.
 By \eqref{1.18} and Lemma~\ref{Lemma 2}, for any $u\in H^1$ we have
 \begin{equation}\label{3.10}
 \begin{aligned}
   I(u)&=\tfrac 12 \|u\|_{H^1}^2-\int_\Omega F(u)-\int_{\Gamma_1}G(u)\\
   &\ge \frac 12 \|u\|_{H^1}^2-\frac {a_0}2\|u\|_{H^0}^2-c_8\left(\|u\|_p^p+\|u\|_{q,\Gamma_1}^q\right).
 \end{aligned}
 \end{equation}
We are now going to apply the generalized Rayleigh formula proved in \cite[Theorem~1.2]{Eigenvalues} (see also \cite{EigenvaluesCorrigendum} for the necessity of the connectedness assumption to assure that $\lambda_1>0$), that is
\begin{equation}\label{3.11}
\lambda_1=\min_{u\in H^1\setminus\{0\}}\frac {\|u\|_{H^1}^2}{\|u\|_{H^0}^2},
\end{equation}
so that we have
\begin{equation}\label{3.12}
\|u\|_{H^1}^2\ge \lambda_1 \|u\|_{H^0}^2\qquad\text{for all $u\in H^1$.}
\end{equation}
By combining \eqref{3.10} with \eqref{3.12} we get the estimate
 \begin{equation}\label{3.13}
   I(u)\ge \frac {\lambda_1-a_0}{2\lambda_1}\|u\|_{H^1}^2-c_8\left(\|u\|_p^p+\|u\|_{q,\Gamma_1}^q\right)\qquad\text{for all $u\in H^1$.}
 \end{equation}
Now, by applying  the  Sobolev Embedding Theorem, both in $\Omega$ and on $\Gamma$, we get the existence of positive constants
$c_{11}=c_{11}(\Omega,p)$ and $c_{12}=c_{12}(\Omega,q)$ such that
\begin{equation}\label{3.13bis}
\|u\|_p\le c_{11} \|u\|_{H^1(\Omega)},\qquad\text{and}\quad \|v\|_{q,\Gamma_1}\le c_{12} \|v\|_{H^1(\Gamma_1)}
\end{equation}
for all $u\in H^1(\Omega)$ and  $v\in H^1(\Gamma_1)$.
hence, using Lemma~\ref{Lemma 1}, the Embedding $H^1\hookrightarrow H^1(\Omega)$ and the bondedness of the Trace Operator from $H^1$ into $H^1(\Gamma_1)$, there is a positive constant $c_{13}=c_{13}(\Omega,p,q)$ such that
\begin{equation}\label{3.14}
\|u\|_p\le c_{13} \|u\|_{H^1},\qquad\text{and}\quad \|u\|_{q,\Gamma_1}\le c_{13} \|u\|_{H^1}\quad\text{for all $u\in H^1$.
}
\end{equation}
By \eqref{3.13} and \eqref{3.14}, we get the estimate
\begin{equation}\label{3.15}
I(u)\ge \|u\|_{H^1}^2\left(\frac {\lambda_1-a_0}{2\lambda_1}-c_8c_{13}^p\|u\|_{H^1}^{p-2}-c_8c_{13}^q\|u\|_{H^1}^{q-2}\right).
\end{equation}
Hence, denoting $\varrho=\|u\|_{H^1}$, since $p,q>2$  and $a_0<\lambda_1$,  \eqref{3.15} yields
\begin{equation}\label{3.16}
\frac{I(u)}{\varrho^2}\ge \frac {\lambda_1-a_0}{2\lambda_1}-c_8c_{13}^p\varrho^{p-2}-c_8c_{13}^q\varrho^{q-2}.
\end{equation}
Now, we fix  $\rho$ so small that
$c_8\left(c_{13}^p\rho^{p-2}+c_{13}^q\rho^{q-2}\right)\le \frac {\lambda_1-a_0}{4\lambda_1}$.
Hence, when $\varrho=\rho$, by \eqref{3.16}  we get $I(u)\ge \eta:=\frac{\lambda_1-a_0}{4\lambda_1}\rho^2>0$,
proving assumption ii) of Theorem~\ref{Theorem 7}. The estimate \eqref{3.16} also proves that, when $\varrho\le \rho$, we have $I(u)>0$, so proving assumption iv) of Theorem~\ref{Theorem 8}.

To check assumption iii) of Theorem~\ref{Theorem 7} we are going to estimate $I$ from above. By \eqref{1.18} and Lemma~\ref{Lemma 2}, \eqref{2.23}, for all $u\in H^1$ we have
$$I(u)\le \tfrac 12 \|u\|_{H^1}^2+c_7\|u\|_{H^0}^2-C_f^+\|u^+\|_{p_0}^{p_0}-C_f^-\|u^-\|_{p_0}^{p_0}
-C_g^+\|u^+\|_{q_0,\Gamma_1}^{q_0}-C_g^-\|u^-\|_{q_0,\Gamma_1}^{q_0}.$$
Now we set
\begin{equation}\label{3.17}
C_f=\min\{C_f^+,C_f^-\}\ge 0\qquad\text{and}\quad C_g=\min\{C_g^+,C_g^-\}\ge 0.
\end{equation}
By using \eqref{3.12} we obtain
\begin{equation}\label{3.18}
I(u)\le \left(\tfrac 12 +\tfrac {c_7}{\lambda_1}\right)\|u\|_{H^1}^2  -C_f\|u\|_{p_0}^{p_0}-C_g\|u\|_{q_0,\Gamma_1}^{q_0}\qquad\text{for all $u\in H^1$.}
\end{equation}
Hence, by choosing $v\in H^1$ such that $v_{|\Gamma_1}\not\equiv 0$ and taking $u=tv$, with $t>0$, in \eqref{3.18}, we get
\begin{equation}\label{3.19}
I(tv)\le \left(\tfrac 12 +\tfrac {c_7}{\lambda_1}\right)\|v\|_{H^1}^2  t^2 -C_f\|v\|_{p_0}^{p_0}t^{p_0}-C_g\|v\|_{q_0,\Gamma_1}^{q_0}t^{q_0}.
\end{equation}
Since $v_{|\Gamma_1}\not\equiv 0$, we have $\|v\|_{p_0}>0$ and $\|v\|_{q_0,\Gamma_1}>0$. Moreover, by assumption (A4),  Lemma~\ref{Lemma 2} and \eqref{3.17}, at least one between $C_f$ and $C_g$ is strictly positive. Then, since $p_0,q_0>2$, by \eqref{3.19} it follows that $I(tv)\to-\infty$ as $t\to\infty$. Hence, choosing $l=tv$ for $t>0$ large enough, we have $\|l\|_{H^1}>\rho$ and $I(l)\le 0$, so assumption iii) of Theorem~\ref{Theorem 7} holds true.

We are finally going to check that, when $\lim_{u\to\pm\infty} \sigma(u)>0$, also  assumption v) of Theorem~\ref{Theorem 9} holds true.
We first notice that, by Lemma~\ref{Lemma 1}, in this case we have $C_f>0$. Since $C_g\ge 0$, by \eqref{3.18} we thus have
\begin{equation}\label{3.20}
I(u)\le c_{14}\|u\|_{H^1}^2  -C_f\|u\|_{p_0}^{p_0}\qquad\text{for all $u\in H^1$,}
\end{equation}
where $c_{14}:=\left(\tfrac 12 +\tfrac {c_7}{\lambda_1}\right)>0$. Now let $Y$ be a finite dimensional subspace of $H^1$.  Since $\|\cdot\|_{p_0}$ is a norm on it, there is a positive constant $c_{15}=c_{15}(p_0,Y)$ such that $\|u\|_{H^1}\le c_{15}\|u\|_{p_0}$ for all $u\in Y$. Then, by \eqref{3.20} we get
$$
I(u)\le \|u\|_{H^1}^2\left(c_{14}  -C_f c_{15}^{-p_0}\|u\|_{H^1}^{p_0-2}\right)\qquad\text{for all $u\in Y$,}
$$
so $I(u)\le 0$ provided $u\in Y$ and $\|u\|_{H^1}\ge R_Y$, where $R_Y:=\left(c_{14}c_{15}^{p_0}/C_f\right)^{\frac 1{p_0-2}}\negquad >0$.
\end{proof}
\begin{lem}\label{Lemma 6} If assumptions (A1--4) hold, $I$ satisfies the (PS) condition.
\end{lem}
\begin{proof}We preliminarily introduce the positively $2$--homogeneous functional $A:H^1\to\R$ defined by
\begin{equation}\label{3.23}
  A(u)=\ell^+\|u^+\|_2^2+\ell^-\|u^-\|_2^2+\mathfrak{m}^+\|u^+\|_{2,\Gamma_1}^2+\mathfrak{m}^-\|u^-\|_{2,\Gamma_1}^2.
\end{equation}
By \eqref{2.19} we have
\begin{equation}\label{3.24}
\begin{alignedat}2
&\int_\Omega F(u)+\int_{\Gamma_1}G(u)=&& \quad\tfrac 12 A(u)+\int_\Omega H(u)+\int_{\Gamma_1} K(u),\\
&\int_\Omega f(u)u+\int_{\Gamma_1}g(u)u=&&\quad A(u)+\int_\Omega h(u)u+\int_{\Gamma_1} k(u)u,
\end{alignedat}\quad\text{for all $u\in H^1$.}
\end{equation}
We also point out that, by \eqref{2.30}, we have
\begin{equation}\label{3.25}
  A(u)\le a_M\|u\|_{H^0}^2\qquad\text{for all $u\in H^0$.}
\end{equation}
We now start by an estimate. Let $\tau_0:=\min\{p_2,q_0\}>2$. By \eqref{1.18} and \eqref{3.9}, for all $u\in H^1$  we have
$$\tau_0 I(u)-\langle dI(u),u\rangle_{H^1}=\left(\frac {\tau_0}2-1\right)\|u\|_{H^1}^2+\int_\Omega \left[f(u)u-\tau_0 F(u)\right]+\int_{\Gamma_1} \left[g(u)u-\tau_0 G(u)\right].$$
Consequently, by \eqref{3.24}, we get
\begin{align*}
\tau_0 I(u)-\langle dI(u),u\rangle_{H^1}=&\left(\frac {\tau_0}2-1\right)\|u\|_{H^1}^2+\left(1-\frac{\tau_0}2\right)A(u)+\int_\Omega \left[h(u)u-\tau_0 H(u)\right]\\
&+\int_{\Gamma_1} \left[k(u)u-\tau_0 K(u)\right]\qquad\text{for all $u\in H^1$.}
\end{align*}
Hence, by  using \eqref{3.25} and then \eqref{3.12}, we obtain
\begin{equation}\label{3.26}
\begin{aligned}
\tau_0 I(u)-\langle dI(u),u\rangle_{H^1}&\ge \left(\frac {\tau_0}2-1\right)\left[\|u\|_{H^1}^2-a_M \|u\|_{H^2}^2\right]\\
&+\int_\Omega \left[h(u)u-\tau_0 H(u)\right]+\int_{\Gamma_1} \left[k(u)u-\tau_0 K(u)\right]\\
&\ge \dfrac {(\tau_0-2)(\lambda_1-a_M)}{2\lambda_1}\|u\|_{H^1}^2+\int_\Omega \left[h(u)u-\tau_0 H(u)\right]\\
&+\int_{\Gamma_1} \left[k(u)u-\tau_0 K(u)\right]\qquad\text{for all $u\in H^1$.}
\end{aligned}
\end{equation}
Recalling the assumption (A3), for any $u\in H^1$ we now introduce the sets
\begin{alignat*}2
&\Omega_u^+=\{x\in\Omega: |u(x)|\ge M\},&&\qquad \Omega_u^-=\{x\in\Omega: |u(x)|< M\},\\
&\Gamma_{1,u}^+=\{x\in\Gamma_1: |u(x)|\ge M\},&&\qquad \Gamma_{1,u}^-=\{x\in\Gamma_1: |u(x)|< M\}.
\end{alignat*}
The first two of them are measurable in $\Omega$ and defined a.e. with respect to the Lebesgue measure, the last two of them verify the same properties with respect to $\mathcal{H}^{N-1}$ on $\Gamma_1$. By using \eqref{1.8} and  \eqref{2.19} in \eqref{3.26}, since $p_0,q_0\ge \tau_0$, we get
\begin{align*}
\tau_0 I(u)-\!\langle dI(u),u\rangle_{H^1}\!&\ge \dfrac {(\tau_0-2)(\lambda_1-a_M)}{2\lambda_1}\|u\|_{H^1}^2+\int_\Omega \left[h(u)u-p_0 H(u)\right]\\
&+\int_{\Gamma_1} \left[k(u)u-q_0 K(u)\right]\\
&\ge \dfrac {(\tau_0-2)(\lambda_1-a_M)}{2\lambda_1}\|u\|_{H^1}^2+\int_{\Omega_u^-} \left[h(u)u-p_0 H(u)\right]\\
&+\int_{\Gamma_{1,u}^-} \left[k(u)u-q_0 K(u)\right]\\
&\ge \dfrac {(\tau_0-2)(\lambda_1-a_M)}{2\lambda_1}\|u\|_{H^1}^2-p_0\int_{\Omega_u^-} H(u)-q_0\int_{\Gamma_{1,u}^-}  K(u)
\end{align*}
for all $u\in H^1$. Consequently, by \eqref{2.18}, we get the main estimate
\begin{equation}\label{3.27}
\begin{aligned}
\tau_0 I(u)-\!\langle dI(u),u\rangle_{H^1}\!
&\ge \dfrac {(\tau_0-2)(\lambda_1-a_M)}{2\lambda_1}\|u\|_{H^1}^2-p_0c_1\int_{\Omega_u^-} (1+|u|^p)\\
&-q_0c_2\int_{\Gamma_{1,u}^-} (1+|u|^q)\\
&\ge \dfrac {(\tau_0-2)(\lambda_1-a_M)}{2\lambda_1}\|u\|_{H^1}^2-c_{16},\qquad\text{for all $u\in H^1$,}
\end{aligned}
\end{equation}
where $c_{16}:=p_0c_1(1+M^p)|\Omega|+q_0c_2(1+M^q)\mathcal{H}^{N-1}(\Gamma_1)$.

 To prove the statement, let $(u_n)_n$ be a (PS) sequence in $H^1$. Then there are nonnegative constants $c_{17}$ and  $c_{18}$, both depending on  the sequence $(u_n)_n$, such that
\begin{equation}\label{3.28}
 |I(u_n)|\le c_{17},\qquad\text{and}\quad |\langle dI(u_n),u_n\rangle_{H^1}|\le c_{18} \|u_n\|_{H^1}\quad\text{for all $n\in\N$.}
\end{equation}
By simply combining \eqref{3.27} and \eqref{3.28} we get that
\begin{align*}
\|u_n\|_{H^1}^2\le &\frac {2\lambda_1}{(\tau_0-2)(\lambda_1-a_M)}\left[\tau_0 I(u_n)-\langle dI(u_n),u_n\rangle_{H^1}+c_{16}\right]\\
\le&\frac {2\lambda_1}{(\tau_0-2)(\lambda_1-a_M)}\left[\tau_0 c_{17}+c_{16}+c_{18}\|u_n\|_{H^1}\right]\quad\text{for all $n\in\N$,}
\end{align*}
which clearly yields that $(u_n)_n$ is bounded in $H^1$. Hence, up to a subsequence, $u_n\rightharpoonup u$ in $H^1$.
The compactness of the operator $dJ$ proved in Lemma~\ref{Lemma 4} yields that, up to a further subsequence,
$dJ(u_n)\to dJ(u)$ in $(H^1)'$, so by the continuity of the Riesz isomorphism we have $\nabla J(u_n)\to \nabla J(u)$. Now, by \eqref{3.9} and \eqref{2.13}, we have $\nabla I(u)=u-\nabla J(u)$ for all $u\in H^1$,
so $u_n=\nabla I(u_n)+\nabla J(u_n)$ and, since $\nabla I(u_n)\to 0$, we get $u_n\to \nabla J(u)$, concluding the proof.
\end{proof}
The following result is the main technical step in the proof of Theorem~\ref{Theorem 1}.
\begin{prop}\label{Proposition 1}Let assumptions (A1-4) hold. Then problem \eqref{1.1} has a weak solution $u\in H^1$ such that $I(u)=c>0$, where $c$ is given by \eqref{1.19bis}.
\end{prop}
\begin{proof} By Lemmas~\ref{Lemma 4}--\ref{Lemma 6}, we can  simply apply Theorem~\ref{Theorem 8}.
\end{proof}
\subsection{Proofs of Theorems~\ref{Theorem 1} and \ref{Theorem 2}} \label{section 3.2}
In this subsection we are going to characterize the Mountain Pass level given by \eqref{1.19bis}. We start by introducing the notation
$H^{1,\pm}_0(\Omega)=\{u\in H^1_0(\Omega): \,\,\pm u\ge 0\,\,\text{a.e. in $\Omega$}\}$,
and the following cones in the space $H^1$,  which depend on the behaviors of $\sigma$ in $[0,\infty)$ and
in $(-\infty,0]$,
\begin{equation}\label{3.29}
 \mathcal{S}=
 \begin{cases}
 \phantom{aa}\{0\},\quad &\text{when $\lim\limits_{u\to\pm \infty}\sigma(u)>0$,}\\
H^{1,-}_0(\Omega),\quad &\text{when $\lim\limits_{u\to\infty} \sigma(u)>0=\lim\limits_{u\to-\infty }\sigma(u)$,}\\
H^{1,+}_0(\Omega),\quad &\text{when $\lim\limits_{u\to-\infty} \sigma(u)>0=\lim\limits_{u\to\infty }\sigma(u)$,}\\
\phantom{a}H^1_0(\Omega),\quad &\text{when $\lim\limits_{u\to\pm \infty}\sigma(u)=0$.}
 \end{cases}
\end{equation}
We shall also denote $H^1_{\mathcal{S}}=H^1\setminus\mathcal{S}$.
\begin{lem}\label{Lemma 7}Let us assume that assumptions (A1--4) hold and set, for any $u\in H^1\setminus\{0\}$, the function $\theta_u:[0,\infty)\to\R$ by $\theta_u(t)=I(tu)$ for all  $t\ge 0$. Then $\theta_u(0)=0$,
    $\theta_u\in C^1([0,\infty)$ and $\theta_u$ can exhibit two different behaviors:
\renewcommand{\labelenumi}{{\roman{enumi})}}
\begin{enumerate}
\item when $u\in\mathcal{S}$ we have $\theta_u'>0$ in $(0,\infty)$ and $\lim\limits_{u\to\infty}\theta_u(t)=\infty$;
\item when $u\in H^1_{\mathcal{S}}$ there is $\tau_u>0$ such that $\theta_u'>0$ in $(0,\tau_u)$, $\theta_u'(\tau_u)=0$, and $\theta'_u<0$ in $(\tau_u,\infty)$, so  $\theta_u(\tau_u)=\max\limits_{[0,\infty)}\theta_u$. Moreover, it is $\lim\limits_{t\to\infty}\theta_u(t)=-\infty$.
\end{enumerate}
\end{lem}
\begin{proof} We fix $u\in H^1\setminus\{0\}$ and consider $t\ge 0$. By \eqref{1.18}, \eqref{2.29} and \eqref{3.23} we have
\begin{equation}\label{3.30}
 \theta_u(t)=\tfrac 12 \left[\|u\|_{H^1}^2-A(u)\right]t^2-\int_\Omega H(tu)-\int_{\Gamma_1}K(tu).
\end{equation}
By \eqref{1.7} one has $\theta_u(0)=0$. By \eqref{2.17} and \eqref{2.18}, the same classical arguments used in Lemma~\ref{Lemma 4} show that $\theta_u\in C^1([0,\infty))$, with
\begin{equation}\label{3.31}
  \theta_u'(t)=\left[\|u\|_{H^1}^2-A(u)\right]t-\int_\Omega h(tu)u-\int_{\Gamma_1}k(tu)u=t\Upsilon_u(t),
\end{equation}
where ($\sigma$ and $\xi$ were continously extended to the whole of $\R$) we denote
\begin{equation}\label{3.32}
  \Upsilon_u(t)=\|u\|_{H^1}^2-A(u)-\int_\Omega \sigma(tu)u^2-\int_{\Gamma_1}\xi(tu)u^2\quad\text{for all $t\ge 0$.}
\end{equation}
Now, recalling that in Remark~\ref{Remark 4} the functions $u\mapsto\sigma'(u)u$ and $u\mapsto\xi'(u)u$ were continuosly extended to the whole of $\R$, and using \eqref{2.18}, the same classical arguments used above show that $\Upsilon_u\in C^1([0,\infty))$, with
\begin{equation}\label{3.33}
 \Upsilon_u'(t)=-\int_\Omega \sigma'(tu)u^3-\int_{\Gamma_1}\xi'(tu)u^3\quad\text{for all $t\ge 0$.}
\end{equation}
We point out that, by \eqref{3.12} and \eqref{3.25}, we have
 \begin{equation}\label{3.34}
  \Upsilon_u(0)=\|u\|_{H^1}^2-A(u)\ge \|u\|_{H^1}^2-a_M\|u\|_{H^0}^2\ge \frac{\lambda_1-a_M}{\lambda_1}\|u\|_{H^1}^2>0.
\end{equation}
Hence $\Upsilon$ is a $C^1$ decreasing function in $[0,\infty)$, with $\Upsilon_u(0)>0$. We are now going to calculate $\lim\limits_{t\to\infty}\Upsilon_u(t)$. With this aim, we first give some estimates for the functions $\Sigma_u,\Theta_u\in C^1(0,\infty))$, appearing in \eqref{3.32},  defined by
\begin{equation}\label{3.35}
  \Sigma_u(t)=\int_\Omega \sigma(tu)u^2, \quad\text{and}\quad \Theta_u(t)=\int_{\Gamma_1}\xi(tu)u^2\quad\text{for all $t\ge 0$.}
\end{equation}
By Lemma~\ref{Lemma 2}, formula \eqref{2.24}, we have
$$\frac{f(v)}{v}\ge -2c_7 +2 C_f^\pm |v|^{p_0-2}, \quad\text{and}\quad
\frac{g(v)}{v}\ge -2c_7 +2 C_g^\pm |v|^{q_0-2}$$
when $\pm v>0$. Hence, by using \eqref{1.12}, there is a nonnegative constant $c_{16}$, depending on $f,g,p,q,p_0$ and $q_0$, such that
\begin{equation}\label{3.36}
  \tfrac 12 \sigma(v)\ge -c_{16}+C_f^\pm |v|^{p_0-2},\,\,\text{and}\,\,
  \tfrac 12 \xi(v)\ge -c_{16}+C_g^\pm |v|^{q_0-2}, \quad\text{when $\pm v\ge 0$.}
\end{equation}
By combining \eqref{3.35} and \eqref{3.36} we thus get the estimates
\begin{equation}\label{3.37}
\begin{aligned}
\tfrac 12 \Sigma_u(t)\ge  &-c_{16}\|u\|_2^2 +\left[C_f^+\|u^+\|_{p_0}^{p_0}+C_f^-\|u^-\|_{p_0}^{p_0}\right]t^{p_0-2},\\
\tfrac 12 \Theta_u(t)\ge  &-c_{16}\|u\|_{2,\Gamma_1}^2 +\left[C_f^+\|u^+\|_{q_0,\Gamma_1}^{q_0}+C_f^-\|u^-\|_{q_0,\Gamma_1}^{q_0}\right]t^{q_0-2}.
\end{aligned}
\end{equation}
We have now to distinguish among the four cases in \eqref{3.29}, starting from the case in which $\lim\limits_{u\to\pm \infty}\sigma(u)>0$. By using Lemma~\ref{Lemma 2}, in this case we have $C_f^\pm>0$, so, by \eqref{3.37}, since $u\not\equiv 0$ and $p_0>2$, in this case we get $\Sigma_u(t)\to \infty$ as $t\to \infty$. Since we have $\Theta_u\ge 0$, by \eqref{3.32} we obtain that
\begin{equation}\label{3.38}
 \Upsilon_u(t)\to-\infty \quad\text{as $t\to\infty$.}
\end{equation}
We now turn to consider the other three remaining cases in \eqref{3.29}, recalling that in all of them, by the assumption (A4), we have $\lim\limits_{u\to\pm \infty}\xi(u)>0$, so $C_g^\pm>0$.
Let us start with the case $\lim\limits_{u\to\infty} \sigma(u)>0=\lim\limits_{u\to-\infty }\sigma(u)$. By Lemma~\ref{Lemma 2} in this case we have $C_f^+>0$ and $\sigma\equiv 0$ in $(-\infty,0]$. Two further subcases may occur:
\renewcommand{\labelenumi}{{\alph{enumi})}}
\begin{enumerate}
\item when $u\in H^1_{\mathcal{S}}$, i.e. $u^+\not\equiv 0$ or $u_{|\Gamma_1}\not\equiv 0$, by \eqref{3.37} we get $\Sigma_u(t)\to \infty$ or $\Theta_u(t)\to \infty$ as $t\to\infty$. Since $\Sigma_u,\Theta_u\ge 0$, by \eqref{3.32} we still get \eqref{3.38};
\item when $u\in\mathcal{S}$, i.e. $u^+\equiv 0$ and $u_{|\Gamma_1}\equiv 0$, since $\sigma\equiv 0$ in $(-\infty,0]$, by \eqref{3.35} we have $\Sigma_u\equiv \Theta_u\equiv 0$, hence by \eqref{3.32} we obtain
    \begin{equation}\label{3.39}
      \Upsilon_u(t)=\Upsilon_u(0)>0\qquad\text{for all $t\ge 0$.}
    \end{equation}
\end{enumerate}
In the case in which $\lim\limits_{u\to-\infty} \sigma(u)>0=\lim\limits_{u\to\infty }\sigma(u)$, by using the same arguments, just considering $u^-$ instead of $u^+$, we get that when $u\in H^1_{\mathcal{S}}$ one obtains \eqref{3.38}, while when $u\in\mathcal{S}$ one gets \eqref{3.39}.
Finally, when $\lim\limits_{u\to\pm \infty}\sigma(u)=0$, we have $\sigma\equiv 0$, so when $u\in H^1_\mathcal{S}=H^1\setminus H^1_0(\Omega)$, we get \eqref{3.38}, while when $u\in H^1_0(\Omega)$ we get \eqref{3.39} again.

The analysis of all cases in \eqref{3.29} then shows that, when $u\in\mathcal{S}$, \eqref{3.39} holds true. Consequently, by \eqref{3.31}, we get $\theta_u'(t)=\Upsilon (0)t>0$ for $t>0$ and also $\theta_u(t)\to\infty$ as $t\to\infty$.
The same analysis also shows that, when $u\in H^1_\mathcal{S}$, \eqref{3.38} holds. Hence, $\Upsilon_u$ being decreasing, setting
$t_u=\sup\{t\ge 0: \Upsilon_u(t)=\Upsilon_u(0)\}$,
$$\tau_u=\sup\{t\ge 0: \Upsilon_u(t)>0)\},\quad \text{and}\quad T_u=\sup\{t\ge 0: \Upsilon_u(t)\ge 0)\},
$$
we have
$0\le t_u < \tau_u\le T_u<\infty$ and
\begin{equation}\label{3.40}
\begin{cases}
\Upsilon_u(t)=\Upsilon_u(0), \quad &\text{for $0\le t\le t_u$},\\
0<\Upsilon_u(t)<\Upsilon_u(0), \quad &\text{for $t_u<t<\tau_u$},\\
\Upsilon_u(t)=0, \quad &\text{for $\tau_u\le t\le T_u$},\\
\Upsilon_u(t)<0, \quad &\text{for $t>T_u$}.
\end{cases}
\end{equation}
We now claim that $\tau_u=T_u$. We suppose by contradiction that $\tau_u<T_u$. Then, by \eqref{3.40}, we have $\Upsilon'(t)=0$ for $t\in (\tau_u,T_u)$, which, by using \eqref{3.33} and the inequalities $\sigma',\xi'\ge 0$, yields that $\sigma'(tu)u^3=0$ a.e. in $\Omega$ and $\xi'(tu)u^3=0$ a.e. on $\Gamma_1$.
Now, by assumption (A4), when $\sigma'(tu)=0$ one necessarily has $\sigma(tu)=0$ and when
$\xi'(tu)=0$ one necessarily has $\xi(tu)=0$. Then we get that
$\sigma(tu)u^2=0$ a.e. in $\Omega$ and $\xi(tu)u^2=0$ a.e. on $\Gamma_1$. Hence, by \eqref{3.32} we have $\Upsilon_u(t)=\|u\|_{H^1}^2-A(u)=\Upsilon_u(0)$, contradicting \eqref{3.40} and proving our claim.

Hence, by \eqref{3.40} and \eqref{3.31}, we get that $\theta_u$ enjoys the monotonicity properties asserted in part ii), and by \eqref{3.31} and \eqref{3.38} we get that $\theta_u'(t)\to -\infty$ as $t\to\infty$, which trivially yields that $\theta_u(t)\to -\infty$ as well as $t\to\infty$, concluding the proof.
\end{proof}
We can now prove Theorems~\ref{Theorem 1} and \ref{Theorem 2}.
\begin{proof}[\bf Proof of Theorem~\ref{Theorem 1}] Let $c$ be the number given in \eqref{1.19bis}. By Proposition~\ref{Proposition 1} we just have to show that $c=d$, where $d$ is given by \eqref{1.19}.
At first we point out that, by Lemma~\ref{Lemma 7}, for all $u\in\mathcal{S}$ we have $\sup_{t>0}I(tu)=\infty$, so
\begin{equation}\label{3.40bis}
  d=\inf_{u\in H^1_\mathcal{S}}\sup_{t>0}I(tu).
\end{equation}
To recognize that $d=c$ we shall prove that $d\ge c$ and $d\le c$. To get the first inequality, let us take $u\in H^1_\mathcal{S}$. By Lemma~\ref{Lemma 7}, ii), there is $\tau_u>0$ such that $\max_{t>0}I(tu)=I(\tau_u u)$ and $I(tu)\to-\infty$ as $t\to \infty$. We now set $\sigma_u\in C([0,1];H^1)$ as follows: $\sigma_u(s)=D_u su$, for all $s\in [0,1]$,
where $D_u>0$ is so large that $I(D_u u)<0$. We then have $\sigma_u\in \Sigma$, so
$\sup_{t>0}I(tu)=I(\tau_u u)=\max_{t>0}I(tu)\ge \max_{s\in [0,1]}\sigma_u(s)\ge c$.
Being $u\in H^1_\mathcal{S}$ arbitrary, we get $d\ge c$.

To prove the reverse inequality we use Proposition~\ref{Proposition 1}. Hence, let $u\in H^1$ be a critical point of $I$ such that $I(u)=c$, so that $dI(u)=0$. Since, recalling the definition of $\theta_u$, we have
\begin{equation}\label{3.41}
  \theta_u'(t)=\langle dI(tu),u\rangle_{H^1} \quad\text{for all $t>0$,}
\end{equation}
we get $\theta_u'(1)=0$. Hence, by Lemma~\ref{Lemma 7}, we get that  $u\in H^1_\mathcal{S}$ and $\tau_u=1$, and consequently $c=I(u)=\max_{t>0}I(tu)\ge d$, concluding the proof.
\end{proof}
The proof of Theorem~\ref{Theorem 2} requires the following slight refinement of \cite[Proposition~6.7, Chapter 6,~p. 98]{ambrosettimalchiodi}.
\begin{lem}\label{Lemma 8}Let $E$ be a Hilbert space and $\mathcal{I}\in C^1(E)$ be such that the functional
$\mathcal{J}:E\to\R$, defined bt $\mathcal{J}(u)=\langle d\mathcal{I}(u),u\rangle_E$ is in $C^1(E)$ as well.
Set the Nehari manifold of $\mathcal{I}$ as
$\mathcal{M}=\{u\in E\setminus\{0\}: \mathcal{J}(u)=0\}$,
and suppose that $\mathcal{M}\not=\emptyset$, that
\begin{gather}\label{3.42}
\text{there is $r_0>0$ such that}\quad u\in\mathcal{M}\Longrightarrow \|u\|_E\ge r_0,\\
\label{3.43} \text{and that}\quad\langle d\mathcal{J}(u),u\rangle_E\not=0\quad\text{for all $u\in\mathcal{M}$.}
\end{gather}
Then $u\in\mathcal{M}$ is a critical point for $\mathcal{I}_{|\mathcal{M}}$ if and only if $d\mathcal{I}(u)=0$.
\end{lem}
\begin{proof}
In the statement of \cite[Proposition~6.7, Chapter 6,~p. 98]{ambrosettimalchiodi},
the functional $\mathcal{I}$ is assumed to belong to $C^2(E)$. Since this regularity is
used only to recognize that $\mathcal{J}\in C^1(E)$, the assumptions $\mathcal{I},\mathcal{J}\in C^1(E)$ are sufficient.
    Moreover, the assumption \cite[(6.4), p. 98]{ambrosettimalchiodi}, which reads as
$d^2\mathcal{J}(u)[u,u]\not=0$ for all $u\in\mathcal{M}$,
is used only to prove that \cite[(6.5), p. 98]{ambrosettimalchiodi} holds. This condition in our setting is nothing but \eqref{3.43}, which was directly assumed here.
\end{proof}
\begin{proof}[\bf Proof of Theorem~\ref{Theorem 2}]
We first claim that \eqref{1.22} holds. By Theorem~\ref{Theorem 1} and Lemma~\ref{Lemma 3}, $I$ has a nontrivial critical point $u$ such that $I(u)=d$. Since, by \eqref{1.18} and \eqref{1.20},
\begin{equation}\label{3.44}
  \mathcal{K}(v)=\langle dI(v),v\rangle_{H^1}\quad\text{for all $v\in H^1$,}
\end{equation}
and since $dI(u)=0$, we have $u\in\mathcal{N}$. Consequently, we get $d=I(u)\ge \inf_{v\in\mathcal{N}}I(v)$. To prove the reverse inequality, let $u\in\mathcal{N}$. By \eqref{1.30}--\eqref{1.31} and \eqref{3.41}, we have $\theta_u'(1)=0$, so by Lemma~\ref{Lemma 7} we have $u\in H^1_\mathcal{S}$ and $\tau_u=1$. Hence $I(u)=\max_{t>0}I(tu)\ge d$. Since $u\in\mathcal{N}$ is arbitrary, we get $\inf_{v\in\mathcal{N}}I(v)\ge d$, proving our claim.

Since, by Lemma~\ref{Lemma 3} and \eqref{3.44}, all nontrivial weak solutions of \eqref{1.1} belong to $\mathcal{N}$, those at leved $d$ are nontrivial lowest energy weak solutions of \eqref{1.1}, and so they satisfy \eqref{1.23}.
Conversely, let $u\in H^1$ satisfy \eqref{1.23}. Then, by \eqref{1.22},  $I(u)=\min_{v\in\mathcal{N}}I(v)$ and consequently $u$ is a critical point for $I_{|\mathcal{N}}$.

We now claim that we can apply Lemma~\ref{Lemma 8} to the functional $I$. Trivially we have $I,\mathcal{K}\in C^1(H^1)$, so we only have to check  that, when $\mathcal{I}=I$ and $\mathcal{J}=\mathcal{K}$, \eqref{3.42} and \eqref{3.43} hold true.
To check \eqref{3.42} we point out that, by \eqref{1.20}, \eqref{2.26}, \eqref{3.12} and \eqref{3.14}, for all $u\in\mathcal{N}$ we have
\begin{align*}
\|u\|_{H^1}^2&=\int_\Omega f(u)u+\int_{\Gamma_1}g(u)u\le a_0\|u\|_{H^0}^2+c_8\left(\|u\|_p^p+\|u\|_{q,\Gamma_1}^q\right)\\
&\le \frac{a_0}{\lambda_1}\|u\|_{H^1}^2+c_8\left(c_{13}^p\|u\|_{H^1}^p+c_{13}^q\|u\|_{H^1}^q\right).
\end{align*}
Consequently, since $u\not\equiv 0$, we have
$c_{13}^p\|u\|_{H^1}^{p-2}+c_{13}^q\|u\|_{H^1}^{q-2}\ge \frac{\lambda_1-a_0}{c_8\lambda_1}>0$,
which immediately yields that $\inf_{u\in \mathcal{N}}\|u\|_{H^1}>0$, so proving \eqref{3.42}.

To prove \eqref{3.43} we point out that, by \eqref{1.20} and \eqref{3.23} we have
\begin{equation}\label{3.45}
 \mathcal{K}(u)=\|u\|_{H^1}^2-A(u)-\int_\Omega h(u)u-\int_{\Gamma_1}k(u)u\quad\text{for all $u\in H^1$},
\end{equation}
so, for all $u,v\in H^1$, we get
\begin{align*}
\langle d\mathcal{K}(u),v\rangle_{H^1}&=2\left[(u,v)_{H^1}-\int_\Omega (\ell^+u^++\ell^-u)v
-\int_{\Gamma_1}(\mathfrak{m}^+u^+ +\mathfrak{m}^-u^-)v\right]\\
&-\int_\Omega h'(u)uv-\int_\Omega h(u)v-\int_{\Gamma_1} k'(u)uv-\int_{\Gamma_1} k(u)v.
\end{align*}
Consequently, using \eqref{3.23} again, for all $u\in H^1$ we have
\begin{equation}\label{3.47}
\begin{aligned}
\langle d\mathcal{K}(u),u\rangle_{H^1}&=2\left[\|u\|_{H^1}^2-A(u)\right]-\int_\Omega h'(u)u^2-\int_{\Gamma_1} k'(u)u^2\\
&-\int_\Omega h(u)u-\int_{\Gamma_1} k(u)u.
\end{aligned}
\end{equation}
By \eqref{3.45} and \eqref{3.47} we get
$$\langle d\mathcal{K}(u),u\rangle_{H^1}=2\mathcal{K}(u)+\int_\Omega h(u)u+\int_{\Gamma_1} k(u)u
-\int_\Omega h'(u)u^2-\int_{\Gamma_1} k'(u)u^2.$$
Consequently, for all $u\in\mathcal{N}$, using \eqref{2.21bis} and the fact that $\mathcal{K}(u)=0$, we get
\begin{equation}\label{3.48}
\begin{aligned}
\langle d\mathcal{K}(u),u\rangle_{H^1}=&\int_\Omega h(u)u+\int_{\Gamma_1} k(u)u
-\int_\Omega h'(u)u^2-\int_{\Gamma_1} k'(u)u^2\\
=&-\int_\Omega \sigma'(u)u^3-\int_{\Gamma_1} \xi'(u)u^3.
\end{aligned}
\end{equation}
Then, supposing by contradiction that $\langle d\mathcal{K}(u),u\rangle_{H^1}=0$, since $\sigma',\xi'\ge 0$, by \eqref{3.48} we get that $\sigma'(u)u^3=0$ a.e. in $\Omega$ and  $\xi'(u)u^3=0$ a.e. on
 $\Gamma_1$. Then, using the assumption (A4), we conclude that $\sigma(u)u^2=0$ a.e. in $\Omega$ and  $\xi(u)u^2=0$ a.e. on  $\Gamma_1$. Hence, as $\mathcal{K}(u)=0$,   \eqref{3.45} yields that
 $\|u\|_{H^1}^2-A(u)=0$. But, by \eqref{3.25}, we have
 $$\|u\|_{H^1}^2-A(u)\ge \|u\|_{H^1}^2-a_M\|u\|_{H^0}^2\ge \frac {\lambda_1-a_M}{\lambda_1}\|u\|_{H^1}^2,$$
 hence (as $a_M<\lambda_1$) we get $u\equiv 0$, this one being the required contradiction.
\end{proof}
\section{Multiplicity} \label{section 4}
\subsection{The main dichotomy} \label{section 4.1}
This section is devote to prove Theorem~\ref{Theorem 3}, so we shall henceforth consider nonlinearities $f$ and $g$ which, beside satisfying assumptions (A1--4), are \emph{odd}. As a consequence, the functions $u\mapsto f(u)/u$ and $u\mapsto g(u)/u$ in assumptions (A2) are even, and we then have
\begin{equation}\label{4.1}
  \ell:=\ell^+=\ell^-<\lambda_1, \qquad\text{and}\quad \mathfrak{m}:=\mathfrak{m}^+=\mathfrak{m}^-<\lambda_1.
\end{equation}
By \eqref{4.1}, since $f$ and $g$ are odd, the functions $h$ and $k$ in \eqref{1.7} are odd as well, and the functions $\sigma$ and $\xi$ is assumption (A4) are even.
Then, recalling the discussion made in Remark~\ref{Remark0} on the possible behaviors of $\sigma$ and $\xi$, we get that the following main alternative holds true:
\renewcommand{\labelenumi}{{\Alph{enumi})}}
\begin{equation}\label{alternative}
\text{A) either $\lim\limits_{|u|\to\infty}\sigma(u)>0$, \quad or  \qquad B) $\sigma\equiv 0$ and $\lim\limits_{|u|\to\infty}\xi(u)>0$.}
\end{equation}
As we are going to see, the proof of Theorem~\ref{Theorem 3} in the case A) is straightforard, since bi Lemma~\ref{Lemma 5} we can directly apply Theorem~\ref{Theorem 9} to the functional I. Unfortunately, in the case B), the functional $I$ does not satisfy the assumption (v) of the just quoted result. This fact can be easily recognized by considering the restriction of $I$ on the space $H^1_0(\Omega)$. Indeed, for any $u$ belonging to this space one has $I(u)=\frac 12\|\nabla u\|_2^2-\frac \ell 2 \|u\|_2^2$ and, by \eqref{3.12}, one has $I(u)>0$ whenever $u\in H^1_0(\Omega)\setminus\{0\}$.
\subsection{The case B)} \label{section 4.2}In the sequel we are then going to consider the case B) in \eqref{alternative}, in which $f(u)=\ell u$ for all $u\in\R$ and problem \eqref{1.1} simplifies to the following one:
\begin{equation}\label{4.2}
\begin{cases} -\Delta u-\ell u=0 \qquad &\text{in
$\Omega$,}\\
\phantom{-}u=0 &\text{on $\Gamma_0$,}\\
-\Delta_\Gamma u +\partial_\nu u =g(u)\qquad
&\text{on
$\Gamma_1$.}
\end{cases}
\end{equation}
For the sake of clearness, we then repeat the assumptions made above in the case B) as they apply to problem \eqref{4.2}:
\begin{itemize}
\item[(A6)] the function $g$ is odd, it  satisfies all requirements in assumptions (A1--4), and we
have $\lim\limits_{|u|\to\infty}\xi(u)>0$ and $\ell<\lambda_1$.
\end{itemize}
By the way weak solutions of \eqref{4.2} are those obtained by particularizing Definition~\ref{Definition 2}.

In order to introduce a suitable variational setting to deal with problem \eqref{4.2}, we set up the following auxiliary nonhomogeneous Dirichlet problem
\begin{equation}\label{4.3}
\begin{cases}
-\Delta u-\lambda u=\zeta&\quad\text{in $\Omega$,}\\
\phantom{-}u=v &\quad\text{on $\Gamma$,}
\end{cases}
\end{equation}
and we make precise what we mean by a weak solution of it.
\begin{definition}\label{Definition 3}
Let $\zeta\in H^{-1}(\Omega)$ and $v\in H^{1/2}(\Gamma)$.
By a \emph{weak solution of problem \eqref{4.3}} we mean $u\in H^1(\Omega)$ such that $u_{|\Gamma}=v$ and $-\Delta u-\lambda u=\zeta$ in $\mathcal{D}'(\Omega)$, or equivalently
\begin{equation}\label{4.4}
  \int_\Omega \nabla u\nabla \phi-\lambda \int_\Omega u\phi=\langle \zeta,\psi\rangle_{H^1_0(\Omega)}\qquad\text{for all $\phi\in H^1_0(\Omega)$.}
\end{equation}
\end{definition}
A more or less standard application of the elliptic theory allows to prove the following well--posedness result for problem \eqref{4.3}, where $\lambda_1^D$ stands for the first eigenvalue of the  Laplacian with homogeneous Dirichlet boundary conditions (see \cite[Chapter~6, p. 356]{evans}).
\begin{lem}\label{Lemma 9}
When $\lambda<\lambda_1^D$, for any $\zeta\in H^{-1}(\Omega)$ and $v\in H^{1/2}(\Gamma)$ problem \eqref{4.3} has a unique weak solution $u\in H^1(\Omega)$. Moreover, there is a positive constnt $c_{17}=c_{17}(\lambda,\Omega)$ such that
\begin{equation}\label{4.5}
  \|u\|_{H^1(\Omega)}\le c_{17}\left(\|\zeta\|_{H^{-1}(\Omega)}+\|v\|_{H^{1/2}(\Gamma)}\right)
\end{equation}
for all $\zeta\in H^{-1}(\Omega)$ and $v\in H^{1/2}(\Gamma)$.
\end{lem}
\begin{proof}We give a proof, only for the reader's convenience, by essentially repeating the proof of
\cite[Chapter~1, \S 1.3, Proposition 1.1, p. 12]{GiraultRaviart}, where the case $\lambda=0$ was considered.
We recall, see \cite[Chapter~1, Theorem~1.5.1.3, p. 38]{grisvard}, that the Trace Operator $\Tr\in\mathcal{L}(H^1(\Omega);H^{1/2}(\Gamma))$ has a right--inverse $\mathcal{R}\in\mathcal{L}(H^{1/2}(\Gamma);H^1(\Omega))$, i.e. one has $\mathcal{R}v_{|\Gamma}=v$ for all $v\in H^{1/2}(\Gamma)$. Using it, one easily see that, given any $v\in H^{1/2}(\Gamma)$, $u\in H^1(\Omega)$ is a weak solution of \eqref{4.3} if and only if $\mathfrak{u}:=u-\mathcal{R}v$ is a weak solution of the following particular case of problem \eqref{4.3}, that is of
\begin{equation}\label{4.6}
\begin{cases}
-\Delta \mathfrak{u}-\lambda \mathfrak{u}=\zeta'&\quad\text{in $\Omega$,}\\
\phantom{-}\mathfrak{u}=0 &\quad\text{on $\Gamma$,}
\end{cases}
\end{equation}
where $\zeta'\in H^{-1}(\Omega)$ is defined  by
\begin{equation}\label{4.7}
\langle \zeta',\varphi\rangle_{H^1_0(\Omega)}=\langle \zeta,\varphi\rangle_{H^1_0(\Omega)}
-\int_\Omega \nabla(\mathcal{R}v)\nabla\varphi+\lambda \int_\Omega \mathcal{R}v\varphi\quad\text{for all $\varphi\in H^1_0(\Omega)$.}
\end{equation}
The bilinear form $a:H^1_0(\Omega)\times H^1_0(\Omega)\to\R$ associated to the weak form of problem \eqref{4.6} is given by
\begin{equation}\label{4.8}
  a[\varphi,\psi]=\int_\Omega \nabla\varphi \nabla \psi -\lambda \int_\Omega \varphi\psi\quad \text{for all $\varphi,\psi\in H^1_0(\Omega)$,}
\end{equation}
and it is trivially continuous. Moreover, since (see \cite[Chapter~6, Theorem~2, p.~356]{evans}) one has
\begin{equation}\label{4.9}
  \lambda_1^D=\min_{u\in H^1_0(\Omega)\setminus\{0\}}\frac {\|\nabla u\|_2^2}{\|u\|_2^2},
\end{equation}
one also has
\begin{equation}\label{4.10}
\|\nabla u\|_2^2\ge \lambda_1^D\|u\|_2^2\quad \text{for all $u\in H^1_0(\Omega)$,}
\end{equation}
and, as it is well--known, $\|\nabla(\cdot)\|_2$  is an equivalent norm in $H^1_0(\Omega)$. Since, by \eqref{4.8} and \eqref{4.10}, we have
$$a[\varphi,\varphi]=\|\nabla\varphi\|_2^2-\lambda \|\varphi\|_2^2\ge \frac{\lambda_D^1-\lambda}{\lambda_1^D}\|\nabla\varphi\|_2\quad\text{for all $\varphi\in H^1_0(\Omega)$,}$$
when $\lambda<\lambda_1^D$ the form $a$ is also coercive.

We can then apply the Lax--Milgram Theorem, see in particular \cite[Chapter~1, \S 1.2,  p. 10]{GiraultRaviart}, to recognize that problem \eqref{4.6} has a unique weak solution $\mathfrak{u}\in H^1_0(\Omega)$ and there is a positive constant $c_{18}=c_{18}(\lambda,\Omega)$ such that
\begin{equation}\label{4.11}
\|\nabla\mathfrak{u}\|_2\le c_{18}\|\zeta'\|_{H^{-1}(\Omega)}\quad\text{for all $\zeta'\in H^{-1}(\Omega)$.}
\end{equation}
Hence, $u:=\mathfrak{u}-\mathcal{R}v$ is the unique weak solution of \eqref{4.3}.using \eqref{4.7} in combination with the boundedness of $\mathcal{R}$, the estimate \eqref{4.5} follows, concluding the proof.
\end{proof}
We are now going to apply Lemma~\ref{Lemma 9} to the auxiliary nonhomogeneous Dirichlet problem
\begin{equation}\label{4.12}
\begin{cases}
-\Delta u-\ell u=0&\quad\text{in $\Omega$,}\\
\phantom{-}u=v &\quad\text{on $\Gamma$,}
\end{cases}
\end{equation}
which is a particular case of problem \eqref{4.3}.
\begin{lem}\label{Lemma 10} When $\ell<\lambda_1$, for any $v\in H^1_{\Gamma_0}(\Gamma)$, problem \eqref{4.12} has a unique weak solution $u\in H^1$. The operator $v\mapsto u$ is linear and bounded from $H^1_{\Gamma_0}(\Gamma)$ into $H^1$, having as range the closed subspace $\mathbb{E}^1$ of $H^1$ given by
\begin{equation}\label{4.13}
 \mathbb{E}^1=\left\{ u\in H^1: \int_\Omega \nabla u\nabla\varphi-\ell\int_\Omega u\varphi=0\quad\text{for all $\varphi\in H^1_0(\Omega)$}\right\}.
\end{equation}
Hence, denoting $u:=\mathbb{D}v$, we get the bijective isomorphism
\begin{equation}\label{4.14}
\mathbb{D}\in\mathcal{L}(H^1_{\Gamma_0}(\Gamma);\mathbb{E}^1),\quad\text{with $\mathbb{D}^{-1}=\Tr_{|\mathbb{E}^1}\in\mathcal{L}(\mathbb{E}^1;H^1_{\Gamma_0}(\Gamma))$.}
\end{equation}
\end{lem}
\begin{proof}
We first notice that, by using \eqref{2.12}, for all $u\in H^1_0(\Omega)\setminus\{0\}$ one has $\|u\|_{H^1}^2/\|u\|_{H^0}^2=\|\nabla u\|_2^2/\|u\|_2^2$. Hence, by \eqref{3.11} and \eqref{4.9} we have
$$\lambda_1=\min_{u\in H^1\setminus\{0\}}\frac {\|u\|_{H^1}^2}{\|u\|_{H^0}^2}\le \min_{u\in H^1_0(\Omega)\setminus\{0\}}\frac {\|\nabla u\|_2^2}{\|u\|_2^2}=\lambda_1^D,$$
so $\ell<\lambda_1^D$ and we can apply Lemma~\ref{Lemma 9} when $\lambda=\ell$ and $\zeta=0$. Moreover, since $v\in H^1_{\Gamma_0}(\Gamma)$, the weak solution $u$ belongs to $H^1$ and, trivially, also to $\mathbb{E}^1$, which is closed subspace of $H^1$. Moreover, there is a positive constant
 $c_{19}=c_{19}(\Omega,\Gamma_0,\Gamma_1)$ such that
 $\|u\|_{H^1}\le c_{19}\|v\|_{H^1(\Gamma)}$ for all $v\in H^1_{\Gamma_0}(\Gamma)$.
 Since the trace operator trivially restricts to $\Tr\in\mathcal{L}(H^1;H^1_{\Gamma_0}(\Gamma))$, the proof is complete.
\end{proof}
The space $\mathbb{E}^1$ introduced in Lemma~\ref{Lemma 10} constitutes a natural constraint for problem \eqref{4.2}, since equations \eqref{4.2}$_1$ and \eqref{4.2}$_2$ automatically hold in it. To write the remaining equation \eqref{4.2}$_3$ in a convenient dual space (equivalently, in a weak form), we recall the realization $-\Delta_{\Gamma_1}^D$ of the Laplace-Beltrami operator introduced in \eqref{2.10}. We also introduce the Dirichlet--to--Neumann operator, associated to the problem \eqref{4.2}$_1$--\eqref{4.2}$_2$. It is the operator $\mathbb{A}\in\mathcal{L}(H^1_{\Gamma_0}(\Gamma);[H^1_{\Gamma_0}(\Gamma)]' )$ defined by
\begin{equation}\label{4.15}
\langle \mathbb{A}v,w\rangle_{H^1_{\Gamma_0}(\Gamma)}=\int_\Omega \nabla(\mathbb{D}v)\nabla(\mathbb{D}w)
-\ell \int_\Omega \mathbb{D}v\,\mathbb{D}w
\quad\text{for all $v,w\in H^1_{\Gamma_0}(\Gamma)$.}
\end{equation}
We notice that, being $\mathbb{D}v\in\mathbb{E}^1$, for all $v,w\in H^1_{\Gamma_0}(\Gamma)$ and  $\phi\in H^1$ such that $\phi_{|\Gamma}=w$, we have $\phi-\mathbb{D}w\in H^1_0(\Omega)$ and then, according to \eqref{4.13}, we have
$$\int_\Omega \nabla(\mathbb{D}v)\nabla(\phi-\mathbb{D}w)-\ell\int_\Omega \mathbb{D}v(\phi-\mathbb{D}w)=0,$$
and consequently, by \eqref{4.15},
\begin{equation}\label{4.16}
\langle \mathbb{A}v,w\rangle_{H^1_{\Gamma_0}(\Gamma)}=\int_\Omega \nabla(\mathbb{D}v)\nabla\phi-\ell\int_\Omega \mathbb{D}v\,\phi.
\end{equation}
The equation \eqref{4.2}$_3$ then assumes the abstract form
\begin{equation}\label{4.17}
 -\Delta_{\Gamma_1}^D v+\mathbb{A}v=g(v)\qquad\text{in $[H^1_{\Gamma_0}(\Gamma)]'$}.
\end{equation}
From what precedes \eqref{4.17} can be equivalently written as
\begin{equation}\label{4.18}
\int_\Omega \nabla(\mathbb{D}v)\nabla(\mathbb{D}\psi)-\ell\int_\Omega \mathbb{D}v\,\mathbb{D}\psi+\int_{\Gamma_1}(\nabla_\Gamma v,\nabla_\Gamma \psi)_\Gamma-\int_{\Gamma_1}g(v)\psi=0
\end{equation}
for all $\psi\in H^1_{\Gamma_0}(\Gamma)$.
The achievements of the preceding discussion are formalized as follows.
\begin{lem}\label{Lemma 11}Let assumption (A6) hods, and $u\in H^1$. Then $u$ is a weak solution of \eqref{4.2} if and only if $u\in \mathbb{E}^1$ and $v=u_{|\Gamma}$ solves equation \eqref{4.17}, or \eqref{4.18}. Moreover, in this case, $u=\mathbb{D}v$.
\end{lem}
\begin{proof} If $u\in H^1$ is a weak solution of \eqref{4.2}, by taking in \eqref{3.3} test functions $\phi\in H^1_0(\Omega)$, and recalling that $f(u)=\ell u$, we immediately get that $u\in\mathbb{E}^1$ and
that $v=u_{|\Gamma}$ satisfies \eqref{4.18}, or equivalently \eqref{4.17}. Conversely, if $v\in H^1_{\Gamma_0}(\Gamma)$ satisfies \eqref{4.18}, by setting $u=\mathbb{D}v\in\mathbb{E}^1$ we trivially have
$v=u_{|\Gamma}$. Moreover, by \eqref{4.18} and \eqref{4.15} we get
$$\langle\mathbb{A}v,\psi\rangle_{H^1_{\Gamma_0}(\Gamma)}+\int_{\Gamma_1}(\nabla_\Gamma v,\nabla_\Gamma \psi)_\Gamma -\int_{\Gamma_1}g(v)\psi=0\quad\text{for all $\psi\in H^1_{\Gamma_0}(\Gamma)$,}$$
which, by using \eqref{4.16}, implies that \eqref{3.2} holds, so $u$ is a weak solution of \eqref{4.2}.
\end{proof}
Equation \eqref{4.18} has a clear variational structure. To make it explicit  we introduce the functional $\Phi:H^1_{\Gamma_0}(\Gamma)\to\R$ given by
\begin{equation}\label{4.19}
 \Phi(v)=\frac 12 \int_\Omega |\nabla(\mathbb{D}v)|^2-\frac \ell 2\int_\Omega |\mathbb{D}v|^2+\frac 12 \int_{\Gamma_1} |\nabla_\Gamma v|_\Gamma^2-\int_{\Gamma_1}G(v).
\end{equation}
The following result points out some trivial properties of $\Phi$ and its connection with the functional $I$ in \eqref{1.18}.
\begin{lem}\label{Lemma 12}  Let assumption (A6) hold. We have $\Phi\in C^1(H^1_{\Gamma_0}(\Gamma))$, its Fr\'{e}chet derivative being given
 by
\begin{equation}\label{4.20}
 \langle d\Phi(v),\psi\rangle_{H^1_{\Gamma_0}(\Gamma)}=\int_\Omega \nabla(\mathbb{D}v)\nabla(\mathbb{D}\psi)-\ell\int_\Omega\mathbb{D}v\,\mathbb{D}\psi +\int_{\Gamma_1}(\nabla_\Gamma v,\nabla_\Gamma \psi)_\Gamma-\int_{\Gamma_1}g(v)\psi
\end{equation}
for all $v,\psi\in H^1_{\Gamma_0}(\Gamma)$.
Consequently, critical points of $\Phi$ coincide with solutions of \eqref{4.17}. Moreover,
$\Phi(0)=0$, $\Phi$ is even and we have $\Phi=I\cdot \mathbb{D}$.
\end{lem}
\begin{proof}
Trivially we have $\Phi(0)=0$ and, being $g$ odd, $G$ is even with $\Phi$. Moreover, recalling that $f(u)=\ell u$, from \eqref{4.19} and \eqref{1.18} the relation  $\Phi=I\cdot \mathbb{D}$ is evident.
Hence, since $I\in C^1(H^1)$, we also have  $\Phi\in C^1(H^1_{\Gamma_0}(\Gamma))$ and \eqref{4.20} holds true. By \eqref{4.20} and \eqref{4.18}, critical points of $\Phi$ and solutions of \eqref{4.17} coincide.
\end{proof}
The following result shows that the functional $\Phi$ satisfies the  remaining geometrical assumptions of Theorem~\ref{Theorem 9}.
\begin{lem}\label{Lemma 13} Let assumption (A6) hold. Then the functional $\Phi$ satisfies assumption ii) in Theorem~\ref{Theorem 7} and assumption v) in Theorem~\ref{Theorem 9}.
\end{lem}
\begin{proof} Checking assumption ii) in Theorem~\ref{Theorem 7} involves estimating $\Phi$ from below. By \eqref{4.19} and Lemma~\ref{Lemma 2}, \eqref{2.25}, for all $v\in H^1_{\Gamma_0}(\Gamma)$ we have
$$\Phi(v)\ge\frac 12 \|\mathbb{D}v\|_{H^1}^2-\frac \ell 2\|\mathbb{D}v\|_2^2-\frac{a_0}2\|v\|_{2,\Gamma_1}^2-c_8\|v\|_{q,\Gamma_1}^q.$$
By \eqref{2.30}, \eqref{2.37} and \eqref{4.1} we have $\ell=\ell^+\le a_M<a_0<\lambda_1$, so the previous estimate yields
$$\Phi(v)\ge\frac 12 \|\mathbb{D}v\|_{H^1}^2-\frac{a_0}2\|\mathbb{D}v\|_{H^0}^2-c_8\|v\|_{q,\Gamma_1}^q.$$
Consequently, by \eqref{3.12} and \eqref{3.13bis}, we get
\begin{equation}\label{4.21}
\Phi(v)\ge\frac{\lambda_1-a_0}{2\lambda_1}\|\mathbb{D}v\|_{H^1}^2-c_8c_{12}^q\|\mathbb{D}v\|_{H^1}^q.
\end{equation}
By Lemma~\ref{Lemma 10} there is  a positive constant $c_{19}=c_{19}(\Omega,\Gamma,\Gamma_0)$ such that
\begin{equation}\label{4.22}
 c_{19}^{-1}\|v\|_{H^1(\Gamma)}\le \|\mathbb{D}v\|_{H^1}\le c_{19}\|v\|_{H^1(\Gamma)}\quad\text{for all $v\in H^1_{\Gamma_0}(\Gamma)$.}
\end{equation}
By combining \eqref{4.21} and \eqref{4.22} we get
$$
\Phi(v)\ge\frac{\lambda_1-a_0}{2\lambda_1c_{19}^2}\|v\|_{H^1(\Gamma)}^2-c_8(c_{12} c_{19})^q\|v\|_{H^1(\Gamma)}^q.$$
Hence, setting $\rho=\left[\frac{\lambda_1-a_0}{4\lambda_1c_9c_{19}^{q+2}c_{12}^q}\right]^{1/(q-2)}$, we have  $\Phi(v)\ge \frac{\lambda_1-a_0}{4\lambda_1c_{19}^2}\|v\|_{H^1(\Gamma)}^2$ whenever $\|v\|_{H^1(\Gamma)}\le\rho$. Consequently, also setting $\eta=\frac{\lambda_1-a_0}{4\lambda_1c_{19}^2}\rho^2$ when $\|v\|_{H^1(\Gamma)}=\rho$ and
assumption ii) in Theorem~\ref{Theorem 7} holds true.

To check assumption v) in Theorem~\ref{Theorem 9} we take any finite dimensional subspace $Y$ of $H^1_{\Gamma_0}(\Gamma)$, Since in $Y$ all norms are equivalent, there is a positive constant $c_{20}=c_{20}(q_0,Y)$ such that
\begin{equation}\label{4.23}
c_{20}^{-1}\|v\|_{q_0,\Gamma_1}\le \|v\|_{H^1(\Gamma)}\le c_{20} \|v\|_{q_0,\Gamma_1}\quad\text{for all $v\in Y$. }
\end{equation}
We now point out that, since $\lim\limits_{|u|\to\infty}\xi(u)>0$, by \eqref{2.28} we have $C_g^\pm>0$. Moreover, since $\xi$ is even, we can set $C_g:=C_g^+=C_g^->0$. By \eqref{4.19} and Lemma~\ref{Lemma 2}, \eqref{2.23}, we consequently get the estimate
\begin{equation}\label{4.24}
\begin{aligned}
\Phi(v)&\le\frac 12 \|\mathbb{D}v\|_{H^1}^2-\frac \ell 2\|\mathbb{D}v\|_2^2+c_7\|v\|_{2,\Gamma_1}^2-C_g\|v\|_{q_0,\Gamma_1}^{q_0}\\
&\le c_{21}  \|\mathbb{D}v\|_{H^1}^2-C_g\|v\|_{q_0,\Gamma_1}^{q_0}\quad\text{for all $v\in H^1_{\Gamma_0}(\Gamma)$,}
\end{aligned}
\end{equation}
where $c_{21}=c_{21}(f,\Omega, \Gamma, \Gamma_0)$ is a positive constant. By combining  \eqref{4.22}--\eqref{4.24}
we then get that
$$\Phi(v)\le \|v\|_{q_0,\Gamma_1}^2\left(c_{19}^2c_{20}^2c_{21}-C_g\|v\|_{q_0,\Gamma_1}^{q_0-2}\right)\quad\text{for all $v\in Y$. }$$
Hence $\Phi(v)\le 0$ provided $\|v\|_{q_0,\Gamma_1}\ge \left(c_{21}c_{19}^2c_{20}^2/C_g\right)^{1/(q_0-2)}>0$, which by \eqref{4.23} completes the proof.
\end{proof}
We are now going to check that $\Phi$ verifies the (PS) condition.
\begin{lem}\label{Lemma 14}  Let assumption A(6)  hold. Then $\Phi$ satisfies the (PS) condition.
\end{lem}
\begin{proof}We shall prove the statement by using  Lemma~\ref{Lemma 6}, where the (PS) condition for the functional $I$ was checked. Having this aim, we start by  making some preliminary remarks concerning the space $H^1$ and the functional $I$.
Being $\mathbb{E}^1=\text{Rank }\mathbb{D}$, one easily checks that $\mathbb{E}^1\cap H^1_0(\Omega)=\{0\}$.
Moreover,  $H^1$ admits the decomposition
\begin{equation}\label{4.25}
  H^1=\mathbb{E}^1\oplus H^1_0(\Omega),
\end{equation}
the respective  projectors $\Pi_{\mathbb{E}^1}:H^1\to\mathbb{E}^1$ and $\Pi_{H^1_0(\Omega)}:H^1\to H^1_0(\Omega)$ being given by
$$\Pi_{\mathbb{E}^1}u=\mathbb{D}u_{|\Gamma}\quad\text{and}\quad \Pi_{H^1_0(\Omega)}u=u-\mathbb{D}u_{|\Gamma}\quad\text{for all $u\in H^1$.} $$
Being projectors, see \cite[Chapter~III,~\S5.4,~p.167]{Kato}, the operators $\Pi_{\mathbb{E}^1}$ and $\Pi_{H^1_0(\Omega)}$ are closed, so by the Closed Graph Theorem we have
\begin{equation}\label{4.26}
\Pi_{\mathbb{E}^1}\in\mathcal{L}(H^1,\mathbb{E}^1),\qquad\text{and}\quad \Pi_{H^1_0(\Omega)}\in\mathcal{L}(H^1,H^1_0(\Omega)).
\end{equation}
Using \eqref{4.25}, given  $\phi, u\in \mathbb{E}^1$ and  $\psi\in H^1_0(\Omega)$, we can rewrite
\eqref{3.9} as follows:
$$\begin{aligned}
 \langle dI(u),\phi+\psi\rangle_{H^1}=&(u,\phi+\psi)_{H^1}-\ell\int_{\Omega} u(\phi+\psi)-\int_{\Gamma_1}g(u)(\phi+\psi)\\
 =& (u,\phi)_{H^1}+\int_{\Omega}\nabla u\nabla \psi-\ell\int_{\Omega}u(\phi+\psi)-\int_{\Gamma_1}g(u)\phi\\
 =&(u,\phi)_{H^1}-\ell\int_{\Omega}u\phi-\int_{\Gamma_1}g(u)\phi
 \end{aligned}
$$
were we used \eqref{4.13} to recognize that $\int_{\Omega}\nabla u\nabla \psi-\ell\int_{\Omega}u\psi=0$.

Consequently, we get that
\begin{equation}\label{4.27}
\langle dI(u),\phi+\psi\rangle_{H^1}=\langle dI_{|\mathbb{E}^1}(u),\phi\rangle_{\mathbb{E}^1}\quad\text{for all $\phi, u\in \mathbb{E}^1$ and  $\psi\in H^1_0(\Omega)$,}
\end{equation}
 where $I_{|\mathbb{E}^1}$, denotes the restriction of $I$ to $\mathbb{E}^1$.

To prove the statement let now $(v_n)_n$  be a (PS) sequence for $\Phi$.
Since $\Phi=I\cdot\mathbb{D} =I_{\mathbb{E}^1}\cdot \mathbb{D}$, and by Lemma~\ref{Lemma 10} $\mathbb{D}$ is a bicontinuous isomorphism,  we get that $(\mathbb{D}v_n)_n$ is a (PS) sequence for
the functional $I_{|\mathbb{E}^1}$ in $\mathbb{E}^1$.
Hence $(I(\mathbb{D}v_n))_n$ is bounded in $\mathbb{E}^1$ and
$dI_{|\mathbb{E}^1}(\mathbb{D}v_n)\to 0$ in $[\mathbb{E}^1]'$ as $n\to\infty$.
But, by \eqref{4.27}, for all $u\in H^1$ we have
\begin{align*}
|\langle dI(\mathbb{D}v_n,u\rangle_{H^1}|&=|\langle dI_{|\mathbb{E}^1}(\mathbb{D}v_n),\Pi_{\mathbb{E}^1}u\rangle_{\mathbb{E}^1}|\\
&\le \|dI_{|\mathbb{E}^1}(\mathbb{D}v_n)\|_{[\mathbb{E}^1]'}\|\Pi_{\mathbb{E}^1}u\|_{H^1}.
\end{align*}
Hence, since $\|\Pi_{\mathbb{E}^1}u\|_{H^1}\le \|\Pi_{\mathbb{E}^1}\|_{\mathcal{L}(H^1;\mathbb{E}^1)}\|u\|_{H^1}$ for  all $u\in H^1$,
we get
$$\|dI(\mathbb{D}v_n)\|_{H^1}\le \|\Pi_{\mathbb{E}^1}\|_{\mathcal{L}(H^1;\mathbb{E}^1)}\|dI_{|\mathbb{E}^1}(\mathbb{D}v_n)\|_{[\mathbb{E}^1]'}$$
as $n\to\infty$, so $(\mathbb{D}v_n)_n$ is actually a (PS) sequence for $I$ in $H^1$.
Hence, by  Lemma~\ref{Lemma 6}, up to a subsequence, $(\mathbb{D}(v_n))_n$ strongly converges in $H^1$ as also in $\mathbb{E}^1$. Consequently, by Lemma~\ref{Lemma 10},  $v_n$ strongly converges in $H^1_{\Gamma_0}(\Gamma)$, completing the proof.
\end{proof}
\subsection{Proof of Theorem~\ref{Theorem 3}} \label{section 4.3}
As we discussed ins \S 4.1, we shall separately consider the cases A) and B) in \eqref{alternative}.

In the case A), since $\lim\limits_{|u|\to\infty}\sigma(u)>0$, we can directly apply Theorem~\ref{Theorem 9} to the functional $I$. Indeed, by Lemma~\ref{Lemma 6}, $I$ satisfies the (PS) condition and, by Lemma~\ref{Lemma 5}, it also satisfies the assumptions i) and ii) of Theorem~\ref{Theorem 7} and the assumption (v) of Theorem~\ref{Theorem 9}. By applying it together with Lemma~\ref{Lemma 3}, we then get the existence of a sequence $(u_n)_n$ of weak solutions of \eqref{1.1} such that $I(u_n)\to \infty$. Since $I$ is even the sequence $(-u_n)_n$ is sequence of weak solutions as well.

In the case B), we Lemmas~\ref{Lemma 11}--\ref{Lemma 12} we can try to apply Theorem~\ref{Theorem 9} to the functional $\Phi$. Actually, by Lemmas~\ref{Lemma 13} and \ref{Lemma 14} its assumptions hold, so completing the proof.

\section{Homogeneous and odd nonlinearities} \label{section 5}
This section is devoted to further characterize $d$ and weak solutions of \eqref{1.1} at this level when $(f,g)$ is odd and positively homogeneous, that is when assumption (A5) holds true.

To prove Theorems~\ref{Theorem 4}--\ref{Theorem 6}, which were  stated in \S~\ref{Section 1.3}, we are going to separately consider the three cases (A5.1), (A5.2) and (A5.3).
\subsection{The case $\boldsymbol{f(u)=\gamma |u|^{p-2}u, \,\,\gamma>0, \,\, 2<p<\romega\,\,\text{and}\,\, g\equiv 0}$}
\label{Section 5.1}
In this subsection we shall consider the case in which assumption (A5.1) holds. In this case, by \eqref{1.18} and \eqref{1.20}, we have
\begin{equation}\label{5.1}
  I(u)=\frac 12 \|u\|_{H^1}^2-\frac \gamma p\|u\|_p^p, \quad\text{and}\quad \mathcal{K}(u)=\|u\|_{H^1}^2-\gamma\|u\|_p^p\quad\text{for all $u\in H^1$.}
\end{equation}
We start with the following key preliminary result.
\begin{lem}\label{Lemma 15}
Let assumption (A5.1) holds. Then \eqref{1.29} holds, and for all $u\in H^1\setminus\{0\}$ such that  $I(u)\le d$, the following implications hold true:
\begin{equation}\label{5.2}
\begin{alignedat}3
&\mathcal{K}(u)\ge 0\quad & \Longleftrightarrow & \quad \|u\|_{H^1}\le\omega_1\quad & \Longleftrightarrow& \quad \|u\|_p\le \omega_2, \\
&\mathcal{K}(u)\le 0\quad & \Longleftrightarrow & \quad \|u\|_{H^1}\ge\omega_1\quad & \Longleftrightarrow& \quad \|u\|_p\ge \omega_2.
\end{alignedat}
\end{equation}
\end{lem}
\begin{proof} We shall use Lemma~\ref{Lemma 7}, keeping the notation in it.
 To prove \eqref{1.29} we point out  that, when (A5.1) holds, we have $\mathcal{S}=\{0\}$, and for all $u\in H^1_\mathcal{S}=H^1\setminus\{0\}$  the function $\Upsilon_u$ in \eqref{3.32} is given by
 $$\Upsilon_u(t)=\|u\|_{H^1}^2-\gamma \|u\|_p^p\,t^{p-2}\qquad\text{for all $t\ge 0$.}$$
 hence
$\tau_u=\|u\|_{H^1}^{-\frac 2{p-2}} \gamma^{-\frac 1{p-2}}\|u\|_p^{-\frac p{p-2}}$.
Consequently, by \eqref{5.1},
$$\max_{\lambda>0}I(\lambda u)=I(\tau_uu)=\left(\tfrac 12 -\tfrac 1p\right)\gamma^{-\frac 2{p-2}}\left(\tfrac {\|u\|_{H^1}}{\|u\|_p}\right)^{\frac {2p}{p-2}}.$$
Then, using \eqref{1.26}, \eqref{1.28} and \eqref{3.40bis}, we get
$$d=\left(\tfrac 12 -\tfrac 1p\right)\gamma^{-\frac 2{p-2}}B_\Omega^{\frac {-2p}{p-2}}=\left(\tfrac 12 -\tfrac 1p\right)\omega_1^2
=\left(\tfrac 12 -\tfrac 1p\right)\gamma\omega_2^p,$$
proving \eqref{1.29}.

Now let $u\in H^1\setminus\{0\}$ such that $I(u)\le d$. To prove \eqref{5.2} we shall first prove the following implications:
\begin{equation}\label{5.3}
\mathcal{K}(u)\ge 0\quad  \Longrightarrow  \quad \|u\|_{H^1}\le\omega_1\quad  \Longrightarrow \quad \|u\|_p\le \omega_2 \quad  \Longrightarrow \quad \mathcal{K}(u)\ge 0.
\end{equation}
When $\mathcal{K}(u)\ge 0$, by \eqref{5.1} we have $\gamma\|u\|_p^p\le \|u\|_{H^1}^2$, and consequently
$$\left(\tfrac 12 -\tfrac 1p\right)\|u\|_{H^1}^2\le \tfrac 12\|u\|_{H^1}^2-\tfrac \gamma p\|u\|_p^p=
I(u)\le d= \left(\tfrac 12 -\tfrac 1p\right)\omega_1^2,$$
which implies that $\|u\|_{H^1}\le \omega_1$.
When $\|u\|_{H^1}\le \omega_1$, by \eqref{1.27} and \eqref{1.28} we get
$\|u\|_p\le B_\Omega \|u\|_{H^1}\le B_\Omega\,\omega_1=\omega_2.$
When $\|u\|_p\le \omega_2$, by \eqref{1.27} and \eqref{1.28} we obtain
$\gamma\|u\|_p^p\le \gamma\omega_2^{p-2}\|u\|_p^2=B_\Omega^{-2}\|u\|_p^2\le \|u\|_{H^1}^2$,
so $\mathcal{K}(u)\ge 0$, concluding the proof of \eqref{5.3}.

To complete the  proof of \eqref{5.2} we are then going to prove the further implications
\begin{equation}\label{5.4}
\mathcal{K}(u)\le 0\quad  \Longrightarrow  \quad \|u\|_p\ge \omega_2\quad  \Longrightarrow \quad \|u\|_{H^1}\ge\omega_1 \quad  \Longrightarrow \quad \mathcal{K}(u)\le 0.
\end{equation}
When $\mathcal{K}(u)\le 0$, by \eqref{5.1}, we have $\gamma \|u\|_p^p\ge \|u\|_{H^1}^2$, so by \eqref{1.27} we obtain
$B_\Omega^{-2}\|u\|_p^2\le \|u\|_{H^1}^2\le \gamma\|u\|_p^p$,
and consequently $\|u\|_p^{p-2}\ge \gamma^{-1} B_\Omega^{-2}=\omega_2^{p-2}$, that is $\|u\|_p\ge \omega_2$.
When $\|u\|_p\ge \omega_2$,  by \eqref{1.27} we get $\|u\|_{H^1}\ge \omega_2B_\Omega^{-1}=\omega_1$.
When $\|u\|_{H^1}\ge\omega_1$, assuming by contradiction that
$\mathcal{K}(u)>0$,  we have $\gamma\|u\|_p^p<\|u\|_{H^1}^2$. Hence, by \eqref{5.1}
$$I(u)=\tfrac 12  \|u\|_{H^1}^2-\tfrac 1p \|u\|_p^p>\left(\tfrac 12 -\tfrac 1p\right)\|u\|_{H^1}^2
\ge \left(\tfrac 12 -\tfrac 1p\right)\omega_1^2=d,$$
the desired contradiction. Hence we have  $\mathcal{K}(u)\le 0$, completing the proof of \eqref{5.4} and consequently also of \eqref{5.2}.
\end{proof}
\begin{proof}[\bf Proof of Theorem~\ref{Theorem 4}]
By Lemma~\ref{Lemma 15}, \eqref{1.29} holds and, for all $u\in H^1\setminus\{0\}$ with $I(u)\le d$, the following implications hold:
\begin{equation}\label{5.5}
\mathcal{K}(u)= 0\quad  \Longleftrightarrow  \quad \|u\|_{H^1}=\omega_1\quad  \Longleftrightarrow \quad \|u\|_p= \omega_2.
\end{equation}
Hence, by Theorem~\ref{Theorem 2}, any lowest energy nontrivial weak solution of \eqref{1.1} satisfies \eqref{1.30} and consequently, since $\omega_2=B_\Omega\omega_1$, it is a solution of the maximization problem \eqref{1.31}.

Conversely, let $u\in H^1\setminus\{0\}$ be a solution of \eqref{1.31}, and let us fix $\tau=\omega_1/\|u\|_{H^1}>0$ and $v=\tau u$. We then have $\|v\|_{H^1}=\omega_1$ and, by \eqref{1.28}, $\|v\|_p=B_\Omega \|v\|_{H^1}=B_\Omega\omega_1=\omega_2$. Moreover, by \eqref{5.1} and \eqref{1.28},
$$I(v)=\tfrac 12 \|v\|_{H^1}^2-\tfrac \gamma p\|u\|_p^p=\tfrac 12 \omega_1^2-\tfrac \gamma p\omega_2^p=
\left(\tfrac 12 -\tfrac 1p\right)\gamma^{-\frac 2{p-2}}B_\Omega^{-\frac{2p}{p-2}}=\left(\tfrac 12 -\tfrac 1p\right)\omega_1^2=d.$$
Hence, by Theorem~\ref{Theorem 2}, as \eqref{1.23} holds true, $v$ is a lowest energy nontrivial weak solution of \eqref{1.1}
\end{proof}

\subsection{The case $\boldsymbol{f\equiv 0\,\,\text{and}\,\, g(u)=\delta |u|^{q-2}u, \,\,\delta>0, \,\, 2<q<\rgamma}$}
\label{Section 5.2}
In this subsection we shall assume that (A5.2) holds, and consequently, for all $u\in H^1$,
\begin{equation}\label{5.6}
  I(u)=\frac 12 \|u\|_{H^1}^2-\frac \delta q\|u\|_{q,\Gamma_1}^q, \quad\text{and}\quad \mathcal{K}(u)=\|u\|_{H^1}^2-\delta\|u\|_{q,\Gamma_1}^q
\end{equation}
The analogous of Lemma~\ref{Lemma 15} in this case is the following one.
\begin{lem}\label{Lemma 16}
Let assumption (A5.2) holds. Then \eqref{1.35} holds, and for all $u\in H^1$ such that $u_{|\Gamma}\not\equiv 0$ and $I(u)\le d$, the following implications hold true:
\begin{equation}\label{5.7}
\begin{alignedat}3
&\mathcal{K}(u)\ge 0\quad & \Longleftrightarrow & \quad \|u\|_{H^1}\le\kappa_1\quad & \Longleftrightarrow& \quad \|u\|_{q,\Gamma_1}\le \kappa_2, \\
&\mathcal{K}(u)\le 0\quad & \Longleftrightarrow & \quad \|u\|_{H^1}\ge\kappa_1\quad & \Longleftrightarrow& \quad \|u\|_{q,\Gamma_1}\ge \kappa_2.
\end{alignedat}
\end{equation}
\end{lem}
\begin{proof} Also in this case we shall use Lemma~\ref{Lemma 7}, keeping the notation.
 To prove \eqref{1.35} we point out  that, when (A5.2) holds, we have $\mathcal{S}=H^1_0(\Omega)$ and,  for any $u\in H^1_\mathcal{S}$, i.e. $u\in H^1$ with $u_{|\Gamma}\not\equiv 0$, the function $\Upsilon_u$ in \eqref{3.32} is given by
 $$\Upsilon_u(t)=\|u\|_{H^1}^2-\delta \|u\|_{q,\Gamma_1}^q\,t^{q-2}\qquad\text{for all $t\ge 0$,}$$
 so
 $\tau_u=\|u\|_{H^1}^{\frac 2{q-2}}\delta^{-\frac 1{q-2}}\|u\|_q^{-\frac q{q-2}}$.
Consequently, by \eqref{5.6},
$$\max_{\lambda>0}I(\lambda u)=I(\tau_uu)=\left(\tfrac 12 -\tfrac 1q\right)\delta^{-\frac 2{q-2}}\left(\tfrac {\|u\|_{H^1}}{\|u\|_{q,\Gamma_1}}\right)^{\frac {2q}{q-2}}.$$
Then, by \eqref{1.32}, \eqref{1.34} and \eqref{3.40bis}, we get
$$d=\left(\tfrac 12 -\tfrac 1q\right)\delta^{-\frac 2{q-2}}B_\Gamma^{\frac {-2q}{q-2}}=\left(\tfrac 12 -\tfrac 1q\right)\kappa_1^2
=\left(\tfrac 12 -\tfrac 1q\right)\delta\kappa_2^p,$$
proving \eqref{1.35}.

Now let $u\in H^1$ such that $u_{|\Gamma}\not\equiv 0$ and $I(u)\le d$. To prove \eqref{5.7} we shall, at  first, prove the following implications:
\begin{equation}\label{5.8}
\mathcal{K}(u)\ge 0\quad  \Longrightarrow  \quad \|u\|_{H^1}\le\kappa_1\quad  \Longrightarrow \quad \|u\|_{q,\Gamma_1}\le \kappa_2 \quad  \Longrightarrow \quad \mathcal{K}(u)\ge 0.
\end{equation}
When $\mathcal{K}(u)\ge 0$, by \eqref{5.6} we have $\delta\|u\|_{q,\Gamma_1}^q\le \|u\|_{H^1}^2$, and consequently
$$\left(\tfrac 12 -\tfrac 1q\right)\|u\|_{H^1}^2\le \tfrac 12\|u\|_{H^1}^2-\tfrac \delta q\|u\|_{q,\Gamma_1}^q=
I(u)\le d= \left(\tfrac 12 -\tfrac 1q\right)\kappa_1^2,$$
which implies that $\|u\|_{H^1}\le \kappa_1$.
When $\|u\|_{H^1}\le \kappa_1$, by \eqref{1.33} and \eqref{1.34} we get
$\|u\|_{q,\Gamma_1}\le B_\Gamma \|u\|_{H^1}\le B_\Gamma\,\kappa_1=\kappa_2$.
When $\|u\|_{q,\Gamma_1}\le \kappa_2$, by \eqref{1.33} and \eqref{1.34} we obtain
$\delta\|u\|_{q,\Gamma_1}^q\le \delta\kappa_2^{q-2}\|u\|_{q,\Gamma_1}^2=B_\Gamma^{-2}\|u\|_q^2\le \|u\|_{H^1}^2$,
so $\mathcal{K}(u)\ge 0$, concluding the proof of \eqref{5.8}.

To complete the  proof of \eqref{5.7} we are then going to prove the further implications
\begin{equation}\label{5.9}
\mathcal{K}(u)\le 0\quad  \Longrightarrow  \quad \|u\|_{q,\Gamma_1}\ge \kappa_2\quad  \Longrightarrow \quad \|u\|_{H^1}\ge\kappa_1 \quad  \Longrightarrow \quad \mathcal{K}(u)\le 0.
\end{equation}
When $\mathcal{K}(u)\le 0$, by \eqref{5.6}, we have $\delta\|u\|_{q,\Gamma_1}^q\ge \|u\|_{H^1}^2$, so by \eqref{1.33} we obtain
$B_\Gamma^{-2}\|u\|_{q,\Gamma_1}^2\le \|u\|_{H^1}^2\le \delta\|u\|_{q,\Gamma_1}^q$,
and consequently $\|u\|_{q,\Gamma_1}^{q-2}\ge \delta^{-1} B_\Gamma^{-2}=\kappa_2^{q-2}$, that is $\|u\|_{q,\Gamma_1}\ge \kappa_2$.
If $\|u\|_{q,\Gamma_1}\ge \kappa_2$,  by \eqref{1.33} we get $\|u\|_{H^1}\ge \kappa_2B_\Gamma^{-1}=\kappa_1$.
When $\|u\|_{H^1}\ge\kappa_1$, assuming by contradiction that
$\mathcal{K}(u)>0$,  we have $\delta\|u\|_{q,\Gamma_1}^q<\|u\|_{H^1}^2$. Hence, by \eqref{5.6},
$$I(u)=\tfrac 12  \|u\|_{H^1}^2-\tfrac 1q \|u\|_{q,\Gamma_1}^q>\left(\tfrac 12 -\tfrac 1q\right)\|u\|_{H^1}^2
\ge \left(\tfrac 12 -\tfrac 1q\right)\kappa_1^2=d,$$
the required contradiction. Hence we have  $\mathcal{K}(u)\le 0$, completing the proof of \eqref{5.9} and consequently also of \eqref{5.7}.
\end{proof}
\begin{proof}[\bf Proof of Theorem~\ref{Theorem 5}]
By Lemma~\ref{Lemma 16}, \eqref{1.35} holds and, for all $u\in H^1$ such that $u_{|\Gamma}\not\equiv 0$ and $I(u)\le d$, the following implications hold:
\begin{equation}\label{5.10}
\mathcal{K}(u)= 0\quad  \Longleftrightarrow  \quad \|u\|_{H^1}=\kappa_1\quad  \Longleftrightarrow \quad \|u\|_{q,\Gamma_1}= \kappa_2.
\end{equation}
Hence, by Theorem~\ref{Theorem 2}, any lowest energy nontrivial weak solution of \eqref{1.1} satisfies \eqref{1.36} and consequently, since $\kappa_2=B_\Gamma\kappa_1$, it is a solution of the maximization problem \eqref{1.37}.

Conversely, let $u\in H^1\setminus\{0\}$ be a solution of \eqref{1.37}. Since the maximum value does not vanish, we have $u_{|\Gamma}\not\equiv 0$. Let us fix $\tau=\kappa_1/\|u\|_{H^1}>0$ and $v=\tau u$. Then we have $\|v\|_{H^1}=\kappa_1$ and $\|v\|_{q,\Gamma_1}=B_\Gamma \|v\|_{H^1}=B_\Gamma\kappa_1=\kappa_2$. Moreover, by \eqref{5.5} and \eqref{1.34},
$$I(v)=\tfrac 12 \|v\|_{H^1}^2-\tfrac \delta q\|u\|_{q,\Gamma_1}^q=\tfrac 12 \kappa_1^2-\tfrac \delta q\kappa_2^2=
\left(\tfrac 12 -\tfrac 1q\right)\delta^{-\frac 2{q-2}}B_\Gamma^{-\frac{2q}{q-2}}=\left(\tfrac 12 -\tfrac 1q\right)\kappa_1^2=d.$$
Hence, by Theorem~\ref{Theorem 2}, as \eqref{1.23} holds true, $v$ is a lowest energy nontrivial weak solution of \eqref{1.1}
\end{proof}
\subsection{The case $\boldsymbol{f(u)=\gamma |u|^{p-2}u, \,\,}$ and $\boldsymbol{g(u)=\delta |u|^{p-2}u,\,\,\gamma,\delta>0, \,\, 2<p<\romega}$}
\label{Section 5.3}
In this subsection we shall consider the case in which assumption (A5.3) holds, hence we have
\begin{equation}\label{5.10bis}
  I(u)=\frac 12 \|u\|_{H^1}^2-\frac \gamma p\|(u,u_{|\Gamma})\|_{X_p}^p, \,\, \mathcal{K}(u)=\|u\|_{H^1}^2-\|(u,u_{|\Gamma})\|_{X_p}^p
\end{equation}
for all $u\in H^1$. In this case our preliminary result takes the following form.
\begin{lem}\label{Lemma 17}
Let assumption (A5.3) holds. Then \eqref{1.42} holds, and for all $u\in H^1\setminus\{0\}$ such that  $I(u)\le d$, the following implications hold true:
\begin{equation}\label{5.11}
\begin{alignedat}3
&\mathcal{K}(u)\ge 0\quad & \Longleftrightarrow & \quad \|u\|_{H^1}\le\zeta_1\quad & \Longleftrightarrow& \quad \|(u,u_{|\Gamma})\|_{X_p}\le \zeta_2, \\
&\mathcal{K}(u)\le 0\quad & \Longleftrightarrow & \quad \|u\|_{H^1}\ge\zeta_1\quad & \Longleftrightarrow& \quad \|(u,u_{|\Gamma})\|_{X_p}\ge \zeta_2.
\end{alignedat}
\end{equation}
\end{lem}
\begin{proof} Also in this case using Lemma~\ref{Lemma 7}, we point out that when
 (A5.3) holds, we have $\mathcal{S}=\{0\}$ and  for all $u\in H^1_\mathcal{S}=H^1\setminus\{0\}$  the function $\Upsilon_u$ in \eqref{3.32} is given by
 $$\Upsilon_u(t)=\|u\|_{H^1}^2-\|(u,u_{|\Gamma})\|_{X_p}^p\,t^{p-2}\qquad\text{for all $t\ge 0$,}$$
 so $\tau_u=\|u\|_{H^1}^{\frac 2{p-2}}\|(u,u_{|\Gamma})\|_{X_p}^{-\frac p{p-2}}$. Consequently, by \eqref{5.10bis},
$$\max_{\lambda>0}I(\lambda u)=I(\tau_uu)=\left(\tfrac 12 -\tfrac 1p\right)\left(\|u\|_{H^1}/\|(u,u_{|\Gamma})\|_{X_p}\right)^{\frac {2p}{p-2}}.$$
Then, using \eqref{1.39} and  \eqref{1.41} we get
$d=\left(\tfrac 12 -\tfrac 1p\right)B_p^{\frac {-2p}{p-2}}=\left(\frac 12 -\frac 1p\right)\zeta_1^2
=\left(\frac 12 -\frac 1p\right)\zeta_2^p$, so proving \eqref{1.42}.

Now let $u\in H^1\setminus\{0\}$ such that $I(u)\le d$. To prove \eqref{5.11} we shall first prove the following implications:
\begin{equation}\label{5.12}
\mathcal{K}(u)\ge 0\,\,  \Longrightarrow  \,\, \|u\|_{H^1}\le\zeta_1\,\,  \Longrightarrow \,\, \|(u,u_{|\Gamma})\|_{X_p}\le \zeta_2 \,\,  \Longrightarrow \quad \mathcal{K}(u)\ge 0.
\end{equation}
When $\mathcal{K}(u)\ge 0$, by \eqref{5.10bis} we have $\|(u,u_{|\Gamma})\|_{X_p}^p\le \|u\|_{H^1}^2$, and consequently
$$\left(\tfrac 12 -\tfrac 1p\right)\|u\|_{H^1}^2\le \tfrac 12\|u\|_{H^1}^2-\tfrac 1 p\|(u,u_{|\Gamma})\|_{X_p}^p=
I(u)\le d= \left(\tfrac 12 -\tfrac 1p\right)\zeta_1^2,$$
so $\|u\|_{H^1}\le \zeta_1$.

When $\|u\|_{H^1}\le \zeta_1$, by \eqref{1.40} and \eqref{1.41} we get
$\|u\|_{X_p}\le B_p \|u\|_{H^1}\le B_p\,\zeta_1=\zeta_2$.
When $\|(u,u_{|\Gamma})\|_{X_p}\le \zeta_2$, by \eqref{1.40} and \eqref{1.41} we obtain
$\|(u,u_{|\Gamma})\|_{X_p}^p\le \zeta_2^{p-2}\|(u,u_{|\Gamma})\|_{X_p}^2=B_p^{-2}\|(u,u_{|\Gamma})\|_{X_p}^2\le \|u\|_{H^1}^2$,
so $\mathcal{K}(u)\ge 0$, concluding the proof of \eqref{5.12}.

To complete the  proof of \eqref{5.11} we are then going to prove the further implications
\begin{equation}\label{5.14}
\mathcal{K}(u)\le 0\,\,  \Longrightarrow  \,\, \|(u,u_{|\Gamma})\|_{X_p}\ge \zeta_2\,\,  \Longrightarrow \quad \|u\|_{H^1}\ge\zeta_1 \,\,  \Longrightarrow \quad \mathcal{K}(u)\le 0.
\end{equation}
When $\mathcal{K}(u)\le 0$, by \eqref{5.10bis}, we have $\|(u,u_{|\Gamma})\|_{X_p}^p\ge \|u\|_{H^1}^2$, so by \eqref{1.40} we obtain
$B_p^{-2}\|(u,u_{|\Gamma})\|_{X_p}^2\le \|u\|_{H^1}^2\le \|(u,u_{|\Gamma})\|_{X_p}^p$,
and consequently $\|(u,u_{|\Gamma})\|_{X_p}^{p-2}\ge B_p^{-2}=\zeta_2^{p-2}$, that is $\|(u,u_{|\Gamma})\|_{X_p}\ge \zeta_2$.
When $\|(u,u_{|\Gamma})\|_{X_p}\ge \zeta_2$,  by \eqref{1.40} we get $\|u\|_{H^1}\ge \zeta_2B_p^{-1}=\zeta_1$.
If $\|u\|_{H^1}\ge\zeta_1$, we assume by contradiction that
$\mathcal{K}(u)>0$. We have $\|(u,u_{|\Gamma})\|_{X_p}^p<\|u\|_{H^1}^2$. Hence, by \eqref{5.10bis}
$$I(u)=\frac 12  \|u\|_{H^1}^2-\frac 1p \|(u,u_{|\Gamma})\|_{X_p}^p>\left(\tfrac 12 -\tfrac 1p\right)\|u\|_{H^1}^2
\ge \left(\tfrac 12 -\tfrac 1p\right)\zeta_1^2=d,$$
the required contradiction. Hence we have  $\mathcal{K}(u)\le 0$, completing the proof of \eqref{5.14} and \eqref{5.11}.
\end{proof}
\begin{proof}[\bf Proof of Theorem~\ref{Theorem 6}]
By Lemma~\ref{Lemma 17}, \eqref{1.42} holds and, for all $u\in H^1\setminus\{0\}$ with $I(u)\le d$, the following implications hold:
\begin{equation}\label{5.14}
\mathcal{K}(u)= 0\quad  \Longleftrightarrow  \quad \|u\|_{H^1}=\zeta_1\quad  \Longleftrightarrow \quad \|(u,u_{|\Gamma})\|_{X_p}= \zeta_2.
\end{equation}
Hence, by Theorem~\ref{Theorem 2}, any lowest energy nontrivial weak solution of \eqref{1.1} satisfies \eqref{1.43} and consequently, since $\zeta_2=B_p\zeta_1$, it is a solution of the maximization problem \eqref{1.44}.

Conversely, let $u\in H^1\setminus\{0\}$ be a solution of \eqref{1.44}, and let us fix $\tau=\zeta_1/\|u\|_{H^1}>0$ and $v=\tau u$. We then have $\|v\|_{H^1}=\zeta_1$ and, by \eqref{1.41}, $\|(u,u_{|\Gamma})\|_{X_p}=B_p \zeta_1=\zeta_2$. Moreover, by \eqref{5.10bis} and \eqref{1.41},
$$I(v)=\tfrac 12 \|v\|_{H^1}^2-\tfrac 1 p\|(u,u_{|\Gamma})\|_{X_p}^p=\tfrac 12 \zeta_1^2-\tfrac 1 p\zeta_2^p=
\left(\tfrac 12 -\tfrac 1p\right)B_p^{-\frac{2p}{p-2}}=\left(\tfrac 12 -\tfrac 1p\right)\zeta_1^2=d.$$
Hence, by Theorem~\ref{Theorem 2}, as \eqref{1.23} holds true, $v$ is a lowest energy nontrivial weak solution of \eqref{1.1}
\end{proof}

\def\cprime{$'$}
\providecommand{\bysame}{\leavevmode\hbox to3em{\hrulefill}\thinspace}
\providecommand{\MR}{\relax\ifhmode\unskip\space\fi MR }
\providecommand{\MRhref}[2]{%
  \href{http://www.ams.org/mathscinet-getitem?mr=#1}{#2}
}
\providecommand{\href}[2]{#2}

\end{document}